\documentclass[12pt]{amsart}
\advance\voffset by -1.2cm

\usepackage{amsmath}
\usepackage{graphicx}
\usepackage{multicol}

\usepackage{txmac,epsf,mypic,times,mathptm}
\makeatletter
\newtheorem{thm}{Theorem}
\newtheorem{cor}[thm]{Corollary}
\newtheorem{pro}[thm]{Proposition}

\theoremstyle{definition}
\newtheorem{lem}[thm]{Lemma}
\newtheorem{dfn}[thm]{Definition}
\newtheorem{ex}[thm]{Example}

\newtheorem{prob}[thm]{Problem}

\newenvironment{ack}{\section*{Acknowledgments}}{}
\def\br#1{\left\langle #1\right\rangle}

\numberwithin{equation}{section}
\numberwithin{thm}{section}

\begin{document}

\let\ul\underline
\let\ol\overline 
\let\reference\ref
\let\dl\delta
\let\Dl\Delta
\let\nb\nabla
\let\sg\sigma
\let\gm\gamma
\let\th\theta 
\let\kp\kappa 
\let\eps\varepsilon
\def\bZ{{\Bbb Z}}
\def\bN{{\mathbb{N}}}
\let\lm\lambda
\let\ap\alpha
\let\bt\beta
\def\lra{\longrightarrow}
\def\llra{\longleftrightarrow}
\let\es\enspace
\def\mc#1{\multicolumn{1}{|c|}{$#1$}}
\def\mC#1{\multicolumn{1}{|c||}{$#1$}}
\def\Y{\ry{1.3em}}
\def\point#1{{\picfillgraycol{0}\picfilledcircle{#1}{0.08}{}}}

\def\epsfsy#1#2{{\catcode`\_=11\relax\ifautoepsf\unitxsize#1\relax\else
\epsfysize#1\relax\fi\epsffile{#2.eps}}}
\def\epsfs#1#2{{\catcode`\_=11\relax\ifautoepsf\unitxsize#1\relax\else
\epsfxsize#1\relax\fi\epsffile{#2.eps}}}
\def\epsfsvy#1#2{{\vcbox{\epsfsy{#1}{#2}}}}
\def\vcbox#1{\setbox\@tempboxa=\hbox{#1}\parbox{\wd\@tempboxa}{\box
  \@tempboxa}}
\def\p{\epsfsv{2cm}}
\def\ang#1{\left\langle#1\right\rangle}
\def\vv#1#2#3{%
  \autoepsffalse
  \begin{tabular}{c}%
  \epsfsvy{4cm}{eps/t1-reiko#1br#2}\\%
  \ry{1.4em}$K_{#3}$\\[3mm]%
  \end{tabular}%
}

\def\@test#1#2#3#4{%
  \let\@tempa\go@
  \@tempdima#1\relax\@tempdimb#3\@tempdima\relax\@tempdima#4\unitxsize
    \relax
  \ifdim \@tempdimb>\z@\relax
    \ifdim \@tempdimb<#2%
      \def\@tempa{\@test{#1}{#2}}%
    \fi
  \fi
  \@tempa
}

\def\go@#1\@end{}
\newdimen\unitxsize
\newif\ifautoepsf\autoepsftrue

\unitxsize4cm\relax
\def\epsfsize#1#2{\epsfxsize\relax\ifautoepsf
  {\@test{#1}{#2}{0.1 }{4   }
		{0.2 }{3   }
		{0.3 }{2   }
		{0.4 }{1.7 }
		{0.5 }{1.5 }
		{0.6 }{1.4 }
		{0.7 }{1.3 }
		{0.8 }{1.2 }
		{0.9 }{1.1 }
		{1.1 }{1.  }
		{1.2 }{0.9 }
		{1.4 }{0.8 }
		{1.6 }{0.75}
		{2.  }{0.7 }
		{2.25}{0.6 }
		{3   }{0.55}
		{5   }{0.5 }
		{10  }{0.33}
		{-1  }{0.25}\@end
		\ea}\ea\epsfxsize\the\@tempdima\relax
		\fi
		}

\ea\def\csname subjclassname@2000\endcsname .%
  {\textup {2000} AMS Subject Classification:}%
% \ea\show\csname subjclassname@2000\endcsname
\subjclass[2000% \noexpand\chnm
]{Primary 57M25, Secondary 20F36, 57M27}

\title{Exchange moves and %non-conjugate\\[2mm]
braid representations of links}

\let\u\\
\author{REIKO SHINJO}
% \tm
\address{Osaka City University Advanced Mathematical Institute, \u
\indent
3-138 Sugimoto 3-chome, Sumiyoshi-ku, Osaka 558-8585, Japan\u[-2mm]}
\email{reiko@suou.waseda.jp\u[-2mm]}

\author{Alexander Stoimenow}
\address{Department of Mathematics, \u
\indent Keimyung University, Daegu 704-701, Korea\u[-2mm]}
{%\parskip\z@
\email{stoimeno@stoimenov.net\u[-18pt]}%
\urladdr{http://stoimenov.net/stoimeno/homepage/}%
}

%%%%%%%%%%%%%%%%%%%%
% Acknowledgements
%%%%%%%%%%%%%%%%%%%
% Use \thanks for acknowledgements as footnotes to the title page.  
% (Note that footnotes inside \author or \title macros are not
% allowed.)
%
% In case of multiple author papers, phrase the acknowledgement to 
% say "The first author was supported by ...  The second author was
% supported by ..."
\thanks{The second author was supported by BK21 Project.}

\begin{abstract}
We prove that under fairly general conditions an iterated exchange
move gives infinitely many non-conjugate braids. As a consequence,
every knot has infinitely many conjugacy classes of $n$-braid
representations if and only if it has one admitting an exchange move. 
\end{abstract}

\maketitle
\section{Introduction}

The \em{braid groups} $B_n$ were introduced in the 1930s in the
work of Artin \cite{Artin}. An element $b\in B_n$ is called
an \em{$n$-braid}. Alexander \cite{Alex} related braids to
\em{links} in real 3-dimensional space, by means of a \em{closure}
operation $\hat{\,\,\,}$. In that realm, in became important to
understand the \em{braid representations} of a given link $L$,
i.e. those $b$ with $L=\hat b$. \em{Markov's theorem} relates these
representations by two moves, the \em{conjugacy} in the braid group,
and \em{(de)stabilization}, which passes between $b\in B_n$ and
$b\sg_{n}^{\pm 1}\in B_{n+1}$ (see e.g. \cite{Morton3}). Conjugacy
is, starting with Garside's \cite{G}, and later many others' work, now
relatively well group-theoretically understood. In contrast, the effect
of (de)stabilization on conjugacy classes of braid representations of
a given link is in general difficult to understand. Only in very special
situations can these conjugacy classes be well described, e.g.
\cite{B-M}.

In this paper we are concerned with the question when infinitely many
conjugacy classes of $n$-braid representations of a given link occur.
Birman and Menasco \cite{B} introduced a move called \em{exchange
move}, and proved that it necessarily underlies the switch between
many conjugacy classes of braid representations of $L$. We will
prove here that it is also sufficient for generating infinitely many
such classes, under a very mild restriction.

\begin{thm}\label{main'}
Let a link $L$ have an $n$-braid representation $b$ admitting
an exchange move, such that the permutation $\pi(b)$ satisfies
$\pi(b)(1)\ne 1$ and $\pi(b)(n)\ne n$. Then iterated exchange
moves on $b$ generate infinitely many non-conjugate braids of $L$.
\end{thm}

Among the different braid representations of a link $L$ the one with
the fewest strands is called a \textit{minimal braid}. The number of
strands of a minimal braid is called the \textit{braid index} $b(L)$
of $L$. Combining our result with the work of Birman and Menasco, we
obtain then:

\begin{cor}\label{main}
Let $L$ be a knot, or a link without trivial components, and let
$n\ge b(L)$. Then $L$ has infinitely many conjugacy classes of $n$-braid
representations if and only if it has one admitting an exchange move. 
\end{cor}

%%%%%%%%%%%%%%%%%%%'a'h'q'l'`'m'Ì—\'z
%%%%%%%%%%%%%%%%%%%%%%%%%%%
Some non-conjugate braids close to isotopic links. Birman had observed
\cite{Birman} that stabilizations of different sign are non-conjugate,
because of different exponent sum. However, the exponent sum is too
weak to distinguish infinitely many conjugacy classes of $n$-braid
representations for any $n$ and $L$. It was also known from \cite{B-M}
that only finitely many conjugacy classes occur when $n\le 3$.

In the case $n>b(L)$ of non-minimal braids, Morton \cite{M}
discovered an infinite sequence of conjugacy classes of 4-braids
with closure being the unknot. Further examples were exhibited more
recently by Fukunaga \cite{F1,F2} and the first author \cite{R}.
For every link, there are obvious (stabilized) non-minimal braid
representations admitting an exchange move. Thus corollary \ref{main}
can always be applied (for knots). The first author obtained this
special case of the theorem in her previous paper \cite{R2}. Her
result for knots was later generalized by the second author to links.
This was done for many links first in \cite{St}, using mostly the
first author's own methods, and later rather completely in \cite{St2},
by an entirely different (Lie group theoretic) approach.

In the case $n=b(L)$ of minimal braids, Birman had conjectured that
there would always be a single conjugacy class of minimal braids
representing a link. However, K.\ Murasugi and R.\ S.\ D.\ Thomas
\cite{Murasugi} disproved Birman's conjecture, exhibiting some
counterexamples in $B_4$. (They claim also such examples in higher
$B_n$, but this is not justified from their proof, which uses the
homomorphism $B_4\to B_3$.) Our result can be seen as such a
construction of nearly exhaustive generality. The few remaining braids
are more subtle, and we discuss them briefly at the end of the paper.

%%%%%%%%%%%%%%%%%%%%%%%%%%%%%%%%%%%%%%%%%%%%%%%%%%%%%%%%%%%%%%%%%%%%%%

\section{Preliminaries}

\subsection{\label{Bc}Braids and closures}

%%%%%%%%%%%% BRAID RERMUTATION %%%%%%%%%%%%%%%%%
\begin{dfn}
The \em{braid group} $B_n$ on $n$ strands can be defined by
\[
B_n\es=\es\br{\,\sg_1,\dots,\sg_{n-1}\es\Bigg|\,
\begin{array}{ll}
[\sg_i,\sg_j]=1 & |i-j|>1 \\
\sg_{j}\sg_i\sg_{j}= \sg_i\sg_{j}\sg_i & |i-j|=1
\end{array}
\,}\,.
\]
The $\sg_i$ are called \em{Artin standard generators}.
An element $b\in B_n$ is an \em{$n$-braid}.
\end{dfn}

There is a graphical representation of braids, where in $\sg_i$
strands $i$ and $i+1$ cross, and multiplication is given by stacking.
(We number strands from left to right and orient them downward.)
\[
\begin{array}{ccl}
\sg_i & = & \diag{8mm}{6}{1}{
  \picvecline{0 1}{0 0}
  \picvecline{1 1}{1 0}
  \picvecline{5 1}{5 0}
  \picvecline{6 1}{6 0}
  \pictranslate{3 0}{
    \picmultivecline{0.18 1 -1.0 0}{0 1}{1 0}
    \picmultivecline{0.18 1 -1.0 0}{1 1}{0 0}
  }
  \picputtext{2.1 0.5}{$\dots$}
  \picputtext{3 1.3}{\scriptsize $i$}
  \picputtext{4 1.3}{\scriptsize $i+1$}
}\\[9mm]
\sg_i^{-1} & = & \diag{8mm}{6}{1}{
  \picvecline{0 1}{0 0}
  \picvecline{1 1}{1 0}
  \picvecline{5 1}{5 0}
  \picvecline{6 1}{6 0}
  \pictranslate{3 0}{
    \picmultivecline{0.18 1 -1.0 0}{1 1}{0 0}
    \picmultivecline{0.18 1 -1.0 0}{0 1}{1 0}
  }
  \picputtext{2.1 0.5}{$\dots$}
  \picputtext{3 1.3}{\scriptsize $i$}
  \picputtext{4 1.3}{\scriptsize $i+1$}
}
\end{array}\qquad\quad
\ap\cdot \bt\quad=\quad\diag{8mm}{3}{4}{
  \picvecline{0.5 4}{0.5 0}
  \picvecline{2.5 4}{2.5 0}
  \picputtext{1.5 3.8}{.\es.\es.}
  \picputtext{1.5 0.2}{.\es.\es.}
  \picfilledbox{1.5 1.2}{3 1.2}{$\bt$}
  \picfilledbox{1.5 2.8}{3 1.2}{$\ap$}
}
\]

The \em{closure} $\hat b$ of a braid $b$ is a knot or
link (with orientation) in $S^3$:
\[
\diag{1cm}{1}{2}{
 \picvecwidth{0.04}
 \pictranslate{-0.5 0}{
  \picvecline{0.75 2}{0.75 0}
  \picvecline{1.0 2}{1.0 0}
  \picvecline{1.25 2}{1.25 0}
  \picfilledbox{1 1}{1 1}{$ b$}
 }%
}\qquad\lra\qquad
%\mbox{
\diag{1cm}{2.5}{3}{
\pictranslate{-0.6 -0.5}{
  \picvecwidth{0.04}
  \picovalbox{2 2}{1.2 2}{0.44}{}
  \picovalbox{2 2}{1.7 2.5}{0.6}{}
  \picovalbox{2 2}{2.2 3}{0.8}{}
  \picvecline{2.6 2}{2.6 2.1}
  \picvecline{2.85 2}{2.85 2.1}
  \picvecline{3.1 2}{3.1 2.1}
  % \piccircle{2 2}{1.8}{}
  \picfilledbox{1.15 2}{1 1}{$ b$}
%}
}}\es=\es\hat b\,.
\]
%\begin{slide}

There is a permutation homomorphism of $B_n$,
\[
\pi\,:\,B_n\to S_n\,,\quad\mbox{given by}\quad
\,\pi(\sg_i)\,=\,(i,\ i+1)\,.
\]
(The permutation on the right is a transposition.) We call 
$\pi(b)$ the \textit{braid permutation} of $b$. We call $b$ a
\em{pure braid} if $\pi(b)=Id$.

Let $b$ be an $n$-braid with numbered endpoints
as in \figurename~\ref{permutation}.
Suppose that $b$ has its strings connected as follows:
$1$ to $i_1$, $2$ to $i_2$, \ldots, $n$ to $i_n$, i.e.
$\pi(b)(k)=i_k$. Then we write
\[
\pi(b)\,=\,\left(
\begin{array}{*4c}
1 & 2  &\ldots & n \\ 
i_1 & i_2 &\ldots  & i_n
\end{array}
\right).
\]

For example the braid $b_1$ in \figurename \ \ref{permutation}
has the permutation
\begin{eqnarray*}
\left(
\begin{array}{*4c}
1 & 2 & 3 & 4\\ 
2 & 4 & 1 & 3
\end{array}
\right)
=
\left(
\begin{array}{*4c}
 1& 2& 4 & 3  
\end{array}
\right),
\end{eqnarray*}
where $(1\hspace{8pt} 2 \hspace{8pt} 4 \hspace{8pt} 3)$ means a
cyclic permutation. The braid $b_2$ in the figure has the permutation 
$(1 \hspace{8pt} 3 \hspace{8pt} 5)(2 \hspace{8pt} 4)$.
We remark that when the closure of $b$ is a $k$-component link,
the braid permutation of $b$ is a product of $k$ disjoint cycles.
The length of each cycle is equal to the number of strings of $b$
which make up a component of $\hat b$.

When we choose a (non-empty) subset $C$ of $\{1,\dots,n\}$ whose
elements form a subset of the cycles of $\pi(b)$, we can define a
\em{subbraid} $b'$ of $b$ by choosing only strings numbered
in $C$. For subbraids $b'$ and $b''$ of $b$ one can define the
\em{(strand) linking number} $lk(b',b'')$ by the linking number
$lk(\hat b',\hat b'')$ between sublinks of $\hat b$. For example,
in $b_2$ of \figurename \ \reference{permutation}, we have
$lk(b_2',b_2'')=0$.

%We remark that the length of the cyclic permutations
%We use cyclic permutations to represent braid permutations.
% braid permutation '©'çŠe¬•ª'ÉŠÜ'Ü
% 'ê'é•R'Ì–{"'ª•ª'©'éB

\begin{figure}[htbp]
\begin{center}
\includegraphics*[scale=1.9]{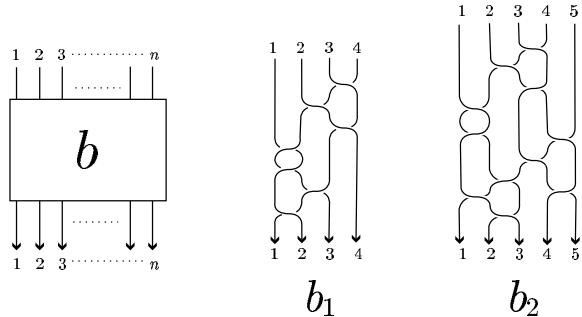}
\end{center}
\caption{An $n$-braid}
\label{permutation}
\end{figure}

%In this paper, we take cyclic permutation
% so that it ended by $n$ for the convenience of argument.
%%%%%%%%%%%%%%%%%%%%% MAIN THEOREM %%%%%%%%%%%%%%%%%%%%%%%%

\subsection{Exchange moves}

Let   
\[
\Dl_n^{2}\,=\,(\sg_1\cdot\ldots\cdot\sg_{n-1})^n
\]
be the (right-handed) full twist on $n$ strands.
The \em{center} of $B_n$ (elements that commute with all $B_n$)
is infinite cyclic and generated by $\Dl_n^{2}$.
Let similarly
\[
\Dl_{[i,j]}^{2}\,=\,(\sg_i\cdot\ldots\cdot\sg_{j-1})^{j-i+1}
\]
be the restricted full twist on strands $i$ to $j$.

We say that $b\in B_n$ \em{admits an exchange move}, if $b$ is as
illustrated in \figurename \ \ref{braid}, where $\alpha, \beta \in
B_{n-1}$. It makes sense to assume $n>3$.

\begin{figure}[htbp]
\begin{center}
  \includegraphics*[scale=1]{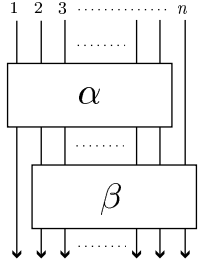}
%\put(-65,60){$\alpha$}
\end{center}
\caption{The $n$-braid $b$.}
\label{braid}
\end{figure}

An \em{exchange move} \cite{B} is the transformation between the
braids $b$ and $b_m$ shown in \figurename \ \ref{bwitht}. Here $m$
is some non-zero integer, and the boxes labeled $\pm m$ represent
the full twists $\Dl_{[2,n-1]}^{\pm 2m}$ respectively,
acting on the middle $n-2$ strands. (Thus a positive
number of full twists are understood to be right full twists, and
$-m$ full twists mean $m$ full left-handed twists.)

\begin{figure}[htbp]
\begin{center}
  \includegraphics*[scale=1]{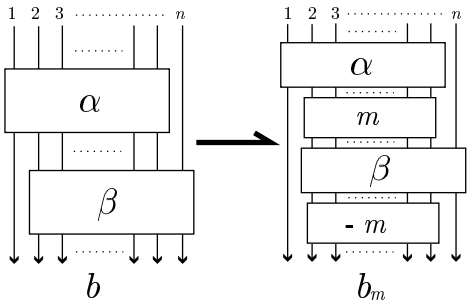}
%\put(-65,60){$\alpha$}
\end{center}
\caption{The braid $b_m$}
\label{bwitht}
\end{figure}

There is another, more common, way to describe the exchange move,
namely by
\begin{equation}\label{alexm}
\ap\bt\,\llra\,\ap\kp^m\bt\kp^{-m}\,,\qquad\mbox{where}
\quad\kp=(\sg_1\cdot\ldots\cdot\sg_{n-2})
(\sg_{n-2}\cdot\ldots\cdot\sg_1)\,.
\end{equation}
This description is equivalent to the previous one, because
$\kp\cdot\Dl_{[2,n-1]}^2=\Dl_{[1,n-1]}^2$, and this element commutes
with $\ap$.

The exchange move underlies the switch between conjugacy classes with
the same closure link, in a universal way.
       
\begin{thm}(Birman-Menasco \cite{B})\label{TMB}
The $n$-braid representations of a given link decompose into a finite
number of classes under the combination of exchange moves and conjugacy.
\end{thm}
	      
\subsection{Axis link and Conway polynomial}

%%%%%%%%%%%%%%%%%AXIS ADDITION LINK%%%%%%%%%%%%%%%%%%%%
% For a braid $b$, we define the \em{axis (addition) link}
% $L_b$ of $b$ as the link consisting of the closure of
% $b$ and its axis.

\begin{dfn}
The \textit{axis (addition) link} of a braid $b$, denoted by $L_{b}$,
is the oriented link 
consisting of the closure of $b$ and an unknotted curve $k$, 
the axis of the closed braid, as in \figurename \ \ref{fig3}.
\end{dfn}

\begin{figure}[htbp]
\begin{center}
  \includegraphics*[scale=0.8]{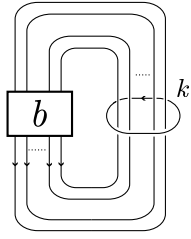}
\end{center}
\caption{The axis-addition link of $b$}
\label{fig3}
\end{figure}

We remark that the axis-addition links of conjugate braids 
are ambient isotopic. Thus for a proof of
non-conjugacy we will study invariants of the axis link. As such an
invariant we will employ the \em{Conway polynomial} $\nb$. It takes
values in $\bZ[z]$ and is given by the value $1$ on the unknot and
the relation
\begin{equation}\label{nbs}
\nb(L_+)-\nb(L_-)=z\nb(L_0)\,.
\end{equation}
This relation involves three links with diagrams
\begin{equation}\label{Lpm0}
\begin{array}{*2{c@{\qquad}}c}
\diag{9mm}{1}{1}{
\picvecwidth{0.09}
\picmultivecline{-8 1 -1.0 0}{1 0}{0 1}
\picmultivecline{-8 1 -1.0 0}{0 0}{1 1}
}
&
\diag{9mm}{1}{1}{
\picvecwidth{0.09}
\picmultivecline{-8 1 -1 0}{0 0}{1 1}
\picmultivecline{-8 1 -1 0}{1 0}{0 1}
}
&
\diag{9mm}{1}{1}{
\picvecwidth{0.09}
\piccirclearc{1.35 0.5}{0.7}{-230 -130}\picvecrlineto{0.01 0.011}
\piccirclearc{-0.35 0.5}{0.7}{310 50}\picvecrlineto{-.01 0.011}
}
\\[2mm]
L_+ & L_- & L_0
\end{array}
\end{equation}
differing just at one crossing. They are called a \em{skein triple}.

%\begin{figure}[htbp]
%\begin{center}
  %\includegraphics*[scale=1]{eps/skein.eps}
%\end{center}
%\caption{Skein trees}
%\label{skein}
%\end{figure}

%The Conway polynomial is an ambient isotopy invariant.
%Since the axis-addition links of 
%conjugate braids are ambient isotopic,
%Theorem \ref{inf} follows from Theorem \ref{mainthm}.

We write $[P]_m$ for the coefficient of $z^m$ in $P\in \bZ[z]$,
and more shortly $a_m=[\nb]_m$ for the coefficient of $z^m$ in
the Conway polynomial.

It is well-known that for an $n$-component link $L$, all coefficients
of $\nb$ in $z$-degree $m$ vanish when $m<n-1$ or $m+n$ is even.

We denote the \em{linking number} of two components of $L$ by
$lk(\cdot,\cdot)$. Now we recall a formula, given by Hoste \cite{H},
which expresses the lowest non-trivial coefficient $a_{n-1}$ of
$\nb(L)$ in terms of component linking numbers.

\begin{thm}\label{th} (see \cite{H})
Let $L=L_1\cup\cdots\cup L_p$ be a $p$-component link of components
$L_1,\dots,L_p$. Let $l_{k,m}=lk(L_k,L_m)$. Then the coefficient
$a_{p-1}$ of the Conway polynomial in degree $p-1$ is
\begin{equation}\label{hf}
a_{p-1}(L)\,=\,\sum_{T}\,\prod_{(k,m)\in T}\,l_{k,m}\,.
\end{equation}
Herein the sum ranges over spanning trees $T$ of the complete
graph $G$ on the vertex set $\{1,\dots,p\}$, and $(k,m)$
denotes the edge in $G$ connecting the $k$-th and $m$-th vertex.
\end{thm}

\proof[Proof of corollary \ref{main}]
The `only if' part in corollary \ref{main} immediately follows
from Theorem \ref{TMB}. The `if' part is a consequence of theorem
\ref{main'}, because under the assumed conditions of $L$, whatever
braid representation $b$ of $L$ satisfies $\pi(b)(k)\ne k$ for
$k=1,n$. \qed

\proof[Proof of theorem \reference{main'}]
We start now the proof of theorem \ref{main'}, which will extend
over several sections until the end of the paper.

In order to exhibit braids $b_m$ in \figurename \ \reference{bwitht}
as non-conjugate, we will follow the approach in \cite{R2},
and evaluate the second coefficient of $\nb$ on the axis addition
link of $b$. We will show that an exchange move alters this
coefficient except in a situation described in the following
proposition.

\begin{pro}\label{l12}
Let $n \ge 4$, and $K$ be a knot represented as the closure of an $n$-%
braid $b$ admitting an exchange move as in \figurename \ \ref{bwitht}.
Write the braid permutation $\pi(b)=(x_1, x_2, \ldots, x_n)$, where we
fix the cyclic ambiguity of $x_i$ by letting the cycle end on $x_n=n$.
Then, \em{unless} $n = 2n'+ 1$ (for $n' \in \bN$) and $x_{n'+1}=1$,
all links $L_{b_m}$, for $b_m$ as in \figurename \ \ref{bwitht}, are
distinguished by $a_3$.
\end{pro}

In the next section we will be concerned with proving this proposition.
The remaining, and more complicated, cases will be settled in \S\ref{S5}
by looking at the axis addition link $L_{b^2}$ of the square of $b$.

\section{Proof of Proposition \ref{l12}\label{S4}}

%%%%%%%%%%%%%%%%%%%%%% LEMMA %%%%%%%%%%%%%%%%%%%%%%%%%%

First, we give a lemma needed later.
A \textit{delta move} is a local move defined in \cite {MN}, and
this move is equivalent to the move in \figurename \ \ref{delta}.
We consider the delta move on the left-hand side in \figurename
\ \ref{local}, where the dotted arcs show how the strands connect.

\begin{figure}[htbp]
\begin{center}
  \includegraphics*[scale=1]{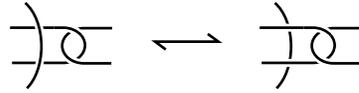}
\end{center}
\caption{A delta move}
\label{delta}
\end{figure}

In a way similar to the proof of Lemma 2.2 in \cite{R2} we can
prove the following lemma using theorem \ref{th}. (We remind that
the linking number and $i$-th coefficient of the Conway polynomial
are written $lk(\cdot,\cdot)$ and $a_i(\cdot)$, respectively.)

%%%%%%%%%%%%%%%%%%%%%%%%%%%%%%%  LEMMA CONWAY   %%%%%%%%%%%%%%%%%%%%%%%%%%%%%%
\begin{lem} \label{a_3}
Let $L$, $L'$ and $l=k_1 \cup k_2 \cup k_3$ be oriented links related 
by the local moves as in \figurename \ \ref {local}.
Then $a_3(L) - a_3(L') = lk(k_2, k_3) - lk(k_3, k_1)$.
\end{lem}

\begin{figure}[htbp]
\begin{center}
  \includegraphics*[scale=1.2]{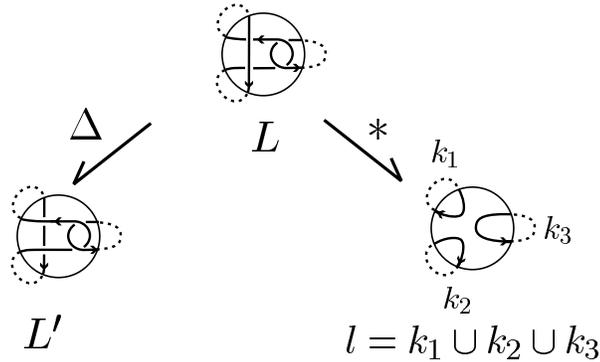}
\end{center}
\caption{Three links related by local moves}
\label{local}
\end{figure}

%\begin{proof}
%From the skein triples as in \figurename \ \ref{skein}, we obtain 
%$a_3(L)-a_3(L') = a_2(L'_0)-a_2(L_0)$, where $L_0=K_1 \cup K_2
%\cup K_3$ and $L%'_0 = K'_1 \cup K'_2 \cup K'_3$ are the $3$-component
%links in the figure.
%It is direct to see  
%\begin{align*}
%lk(K_1,K_2)&=lk(K'_1,K'_2)=lk(k_1,k_2),\\
%lk(K_2,K_3)&=lk(K'_2,K'_3)+1=lk(k_2, k_3)+1,\\
%lk(K_3,K_1)&=lk(K'_3, K'_1)-1 = lk(k_3, k_1).
%\end{align*}
%Notice that 
%\begin{align*}
%&a_2(L_0)\\
%&= lk(K_1,K_2)lk(K_2,K_3)+lk(K_3,K_1)lk(K_1,K_2)+lk(K_2,K_3)lk(K_3,K_1),\\
%&a_2(L'_0)\\
%&= lk(K'_1,K'_2)lk(K'_2,K'_3)+lk(K'_3,K'_1)lk(K'_1,K'_2)+lk(K'_2,K'_3)lk(K'_3,K%'_1).
%\end{align*}
%which are due to \cite{H}.
%From these equations, we have the desired equation $a_3(L) - a_3(L') = lk(k_2, k%_3) - lk(k_3, k_1)$.
%\end{proof}

Now we are ready to start the proof of the first
partial case of Theorem \ref{main'}.

\begin{proof}[Proof of Proposition \ref{l12}]
We set $b_0=b$.
From a braid $b$, we construct an infinite sequence of braids 
$\{b_m, m \in \bZ\}$ 
as in \figurename \ \ref{bwitht} where $m$ and $-m$ represent 
$m$ and $-m$ full twists respectively.
The closures of $b_m$ and $b_{m-1}$ ($m \in \bN$) 
are related by ambient isotopy as in 
\figurename \ \ref{bm} where $A$ denotes the 
braid axis.
This means that all braids in the sequence close to $K$.

\begin{figure}[htbp]
\begin{center}
  \includegraphics*[scale=1.1]{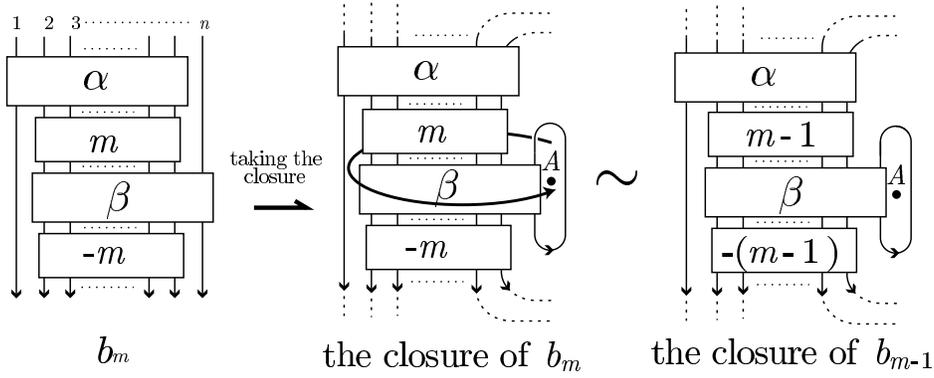}
\end{center}
\caption{Braids with the same closure}
\label{bm}
\end{figure}
%We remark that a half twist of an $n$-braid is given by 
%\begin{align*}
%(\sigma_{n-1}\sigma_{n-2}\ldots\sigma_2\sigma_1)
%(\sigma_{n-1}\sigma_{n-2}\ldots\sigma_3\sigma_2) 
%\ldots
%(\sigma_{n-1}\sigma_{n-2})\sigma_{n-1}   
%\end{align*}
%and that half twist of $n$ strings has 
%the same form as a half twist of an $n$-braid
%when we ignore the orientation of the strings.
Since a full twist of $n$ strings can be deformed as in 
\figurename \ \ref{atwist} up to ambient isotopy,
the axis addition link $L_{b_m}$ of $b_m$, which is the
leftmost diagram in \figurename \ \ref{mtwists},
can be deformed into the rightmost link  
in the figure, still denoted by $L_{b_m}$. 
Here $k$ is the component corresponding 
to the braid axis and 
the boxes $m$ and $-m$ represent $m$-full twists and $-m$-full
 twists respectively.

\begin{figure}[htbp]
\begin{center}
  \includegraphics*[scale=0.8]{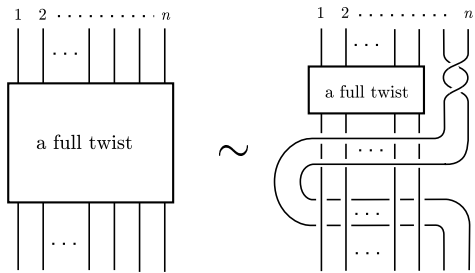}
\end{center}
\caption{A full twist of $n$-strings}
\label{atwist}
\end{figure}

\begin{figure}[htbp]
\begin{center}
  \includegraphics*[scale=0.9]{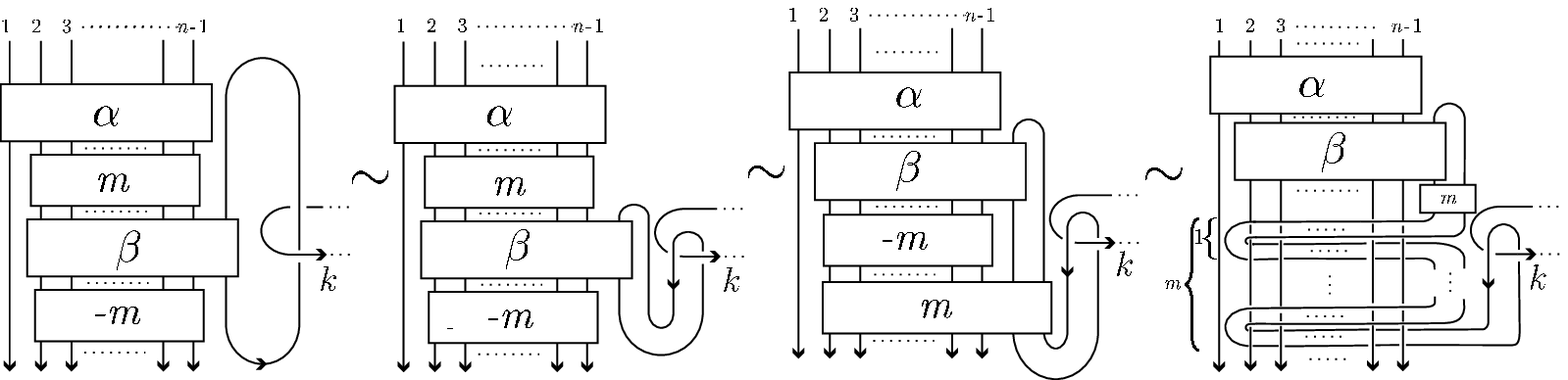}
\end{center}
\caption{$L_{b_m}$}
\label{mtwists}
\end{figure}

Then there are  sequences of links $L_{b_m}= L^0, L^1, L^2,
\ldots, L^{n-1}=L_{b_{m-1}}$ and $l^0,l^1,l^2, \ldots,l^{n-1}$ such
that $L^{i+1}$ and $l^i$ are obtained from $L^i$ by  
the delta move $\Delta_i$ and the move $*_i$  
illustrated in \figurename \ \ref{delta0} ($i=0$) and \ref{deltak}
($i=1, \ldots, n-2$).
%such that $L^1$ is obtained from $L^0$ 
%by the delta move  $\Delta_0$
%illustrated in \figurename \ \ref{delta0}
%and $L^{i+1}$ is obtained from $L^i$ by  
%thebdelta move $\Delta_i$  
%illustrated in \figurename \ \ref{deltak} 
%($i=1,2, \ldots, n-2$).
%Note that these moves can be realized by the local moves in \figurename \ \ref{%local}.
By Lemma \ref{a_3},  
%the change in $a_3$ 
%resulting from  $\Delta_i$
%illustrated in the \figurename~\ref{deltak}
%($i=0,1, \ldots, n-2$) is determined by $lk(l_1^i \cup l_3^i)$ and 
%$lk(l_2^i \cup l_3^i)$.
%Hence 
the change in $a_3$ resulting from $\Delta_0$ can be obtained 
as follows:
\begin{align*}
a_3(L^1)-a_3(L^0)=lk(l_1^0 \cup l_3^0) - lk(l_2^0 \cup l_3^0) = n-1,
\end{align*}
where $l^0=l_1^0 \cup l_2^0 \cup l_3^0$ is the $3$-component 
link illustrated in \figurename~\ref{delta0}.

\begin{figure}[htbp]
\begin{center}
  \includegraphics*[scale=1.2]{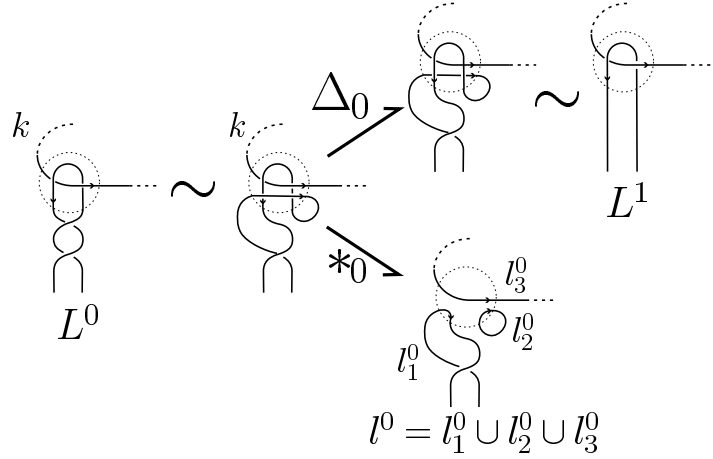}
\end{center}
\caption{The moves $\Delta_0$ and $*_0$}
\label{delta0}
\end{figure}

\begin{figure}[htbp]
\begin{center}
  \includegraphics*[scale=1.3]{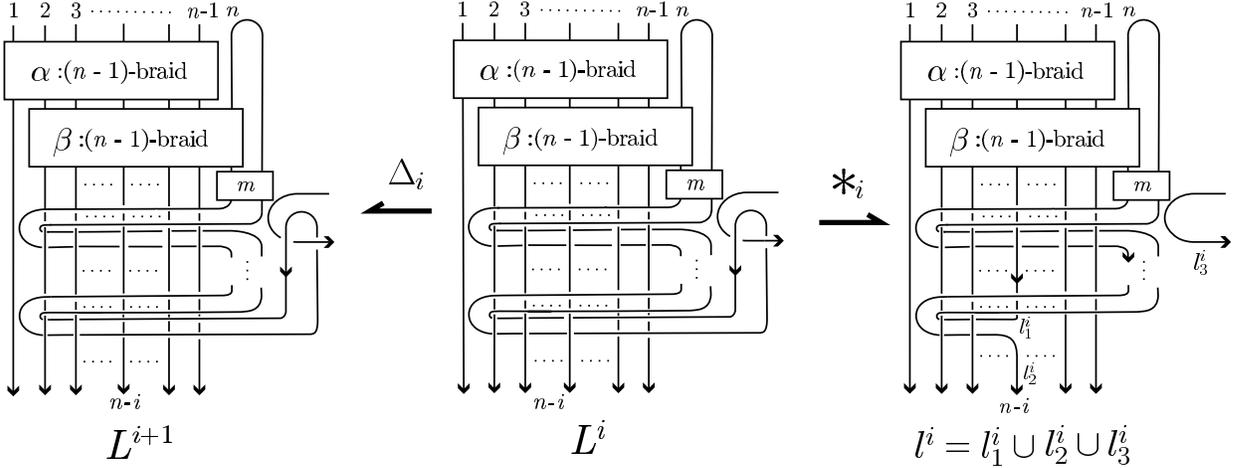}
\end{center}
\caption{The local moves on $L^{i}$}
\label{deltak}
\end{figure}

Next we consider the change in $a_3$ resulting from  $\Delta_i$
illustrated in the \figurename~\ref{deltak} ($i=1,2, \ldots, n-2$).
Let $S_{L^i}$(resp. $S_{l^i}$) be a part of $L^i$ (resp. $l^i$) as
in the left (resp. right) diagram of \figurename~\ref{deltak2}.
Namely $S_{L^i}$ and $S_{l^i}$ are the unions of $n$ strings and an 
unknotted component. Some of these $(n-1)$ strings of $S_{l^i}$
belong to $l_1^i$ and the other belong to $l_2^i$. The numbers of
strings determine $lk(l_1^i \cup l_3^i)$ and $lk(l_2^i \cup l_3^i)$.
% •Ê'ê•û'Í'g'Ý•R'uŠ·
% 'Ì'Ý'É'æ'Á'ÄŒˆ'Ü'éB
%Hence we examine the number of strings 
%by using braid permutations.
%\begin{figure}[htbp]
%\begin{center}
%  \includegraphics*[scale=1.5]{eps/linking.eps}
%\put(-80,-20){\LARGE{$l^i$}}
%\put(-325,-20){\LARGE{$L^i$}}
%\end{center}
%\caption{}
%\label{l_i}
%\end{figure}
%In \figurename~{deltak2},
%how $\ast^i$ change the $n$ strings is connected by 
%We examine how the permutation changes by $\ast_i$.

By considering how $S_{l^i}-l_3^i$ has its strings connected,
permutations of the $n$ down going strings can be assigned
to $S_{L^i}$ and $S_{l^i}$, similarly to a braid permutation.
We call these the permutations of $S_{L^i}$ and $S_{l^i}$.
Note that the permutation assigned to $S_{L^i}$ is the same 
as the braid permutation $\pi(b)$ of $b$. Since $l^i-l_3^i$
is a $2$ component link, the permutation of $S_{l^i}$ consists
of $2$ cycles.
% By \figurename~\ref{deltak2}
To determine the length of these cycles, we observe that
the move $*_i$ corresponds to taking the product of a
transposition $(n-i, n)$ with the permutation of $S_{L^i}$.

\begin{figure}[htbp]
\begin{center}
  \includegraphics*[scale=1.2]{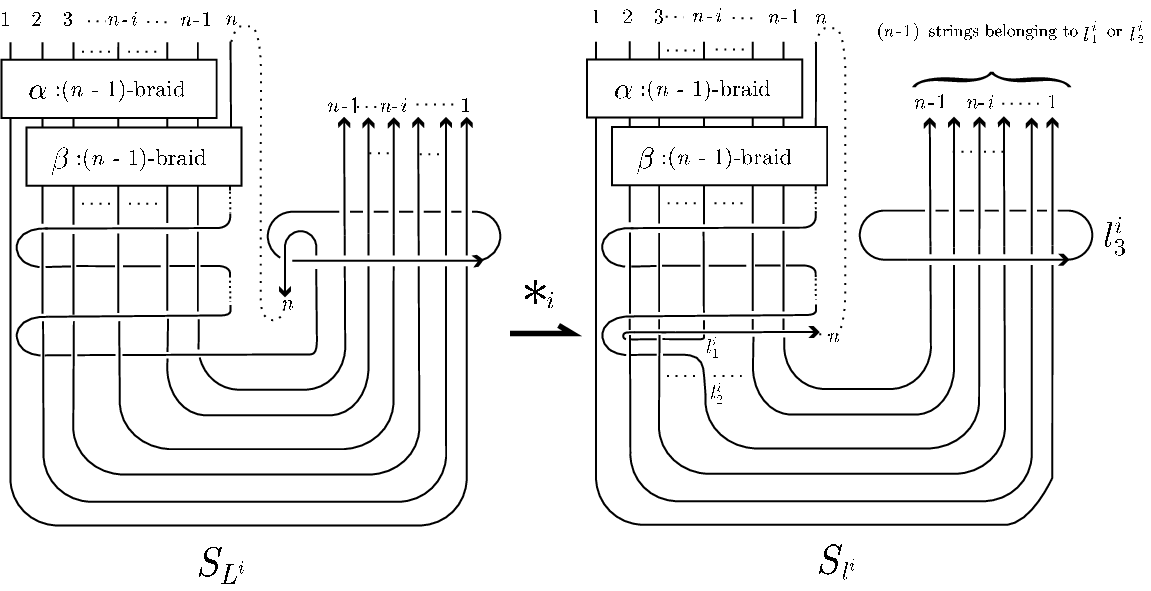}
%\put(-270,-30){$(x_1, x_2, \ldots, x_{n-1}, n) 
%\hspace{20pt} \longrightarrow \hspace{20pt} (x_1, x_2, \ldots, x_{j-1}, n)(x_j,% \ldots ,x_{n-2}, x_{n-1})$}
\end{center}
\caption{}
\label{deltak2}
\end{figure}

Let $n-i = x_j$. Then 
\begin{align*}
(x_j, n)(x_1, x_2, \ldots, x_{n-1}, n)\
= (x_1, x_2, \ldots, x_{j-1}, n)(x_j, \ldots ,x_{n-2}, x_{n-1}).
\end{align*}
%From the \figurename @@@@, we have that 
The cyclic permutations $(x_1, x_2, \ldots, x_{j-1}, n)$ and 
$(x_j, \ldots ,x_{n-2}, x_{n-1})$
correspond to $l_1^i$ and $l_2^i$, respectively.
Remark that the string of $S_{l^i}$ with lower end point 
$n$ belongs to $l^i_1$, and it does not contribute now
to $lk(l_1^i \cup l_3^i)$. By Lemma \ref{a_3}, 
\begin{align*}
a_3(L^i)-a_3(L^{i-1})&=lk(l_1^i \cup l_3^i) - lk(l_2^i \cup l_3^i) \\ 
&= (j-1) - (n-j) = 2j - n - 1.
\end{align*}

Suppose that $x_l = 1$, then \\[2mm]
\begin{align*}
a_3(L_{b_{m+1}})-a_3(L_{b_{m}}) &= \{ \text{the change in} \ a_3 \
 \text{resulting from} \ \Delta_0 \}\\
&\quad \quad \quad \quad  \quad \quad \quad+ \sum _{k=1}^{n-2}
 \{ \text{the change in} \ a_3 \ \text{resulting from} \ \Delta_k \}\\
&=(n-1) + \sum _{j=1}^{n-1}(2j -n -1)-(2l - n-1)= -2l+n+1.
\end{align*}

The difference $-2l+n+1$ is a constant which does not depend on $m$.
If it is non-zero, the sequence $\{a_3(L_{b_p}), p \in \bN \}$ forms 
an arithmetic progression with non zero common difference.
When $n$ is even, $-2l+n+1$ is odd. This means that $-2l+n+1 \neq 0$. 
When $n$ is odd, namely $n = 2n'+1$ for some $n' \in \bN$,
then $-2l+n+1= 2(n'-l+1)$. Unless $l=n'+1$, we have $-2l+n+1 \neq 0$.  
The equation $l=n'+1$ means that $x_{n'+1}=1$. Therefore
$a_3(L_{b_{m+1}})-a_3(L_{b_{m}})$ is non-zero and independent of $m$,
unless $n=2n'+1\es (n' \in \bN)$ and $x_{n'+1}=1$.
This completes the proof of Proposition \ref{l12}. 
\end{proof}

\section{Remaining knot cases \label{S5}}

{}From now on we assume that $n = 2n'+1$ and $\pi(b)=(x_1, x_2,
\ldots, x_{n-1}, n)$ with $x_{n'+1}=1$. To prove that $b_m$ are
non-conjugate, we will look at $b^2_m$: if two braids
are conjugate, so are their squares. Note that, when $n$ is odd
and $\pi(b)$ is a cycle, so is $\pi(b^2)$. Thus $L_{b^2_m}$ are
again 2-component links. We will show the following:

\begin{pro}\label{l13}
Let $b$ be an $n$-braid admitting an exchange move.
If $n = 2n'+1>3$ odd and $\pi(b)=(x_1, x_2, \ldots, x_{n-1}, n)$ with
$x_{n'+1}=1$, then $a_3(L_{b^2_m})$ is a quadratic polynomial in $m$
with non-zero quadratic term.
\end{pro}

In particular, there are at most two $L_{b^2_m}$ with equal $a_3$,
and so at most two of $b_m$ are conjugate. Thus with proposition
\ref{l13}, the proof of theorem \ref{main'} for knots will be complete.

\begin{proof}[Proof of proposition \ref{l13}]
Let us first simplify the form of $\ap$ and $\bt$ in 
\figurename \ \ref{braid}.

First, every permutation of $2,\dots,n-1$ applied on either side of
$\pi(\ap)$ can be realized by a braid which can be moved into $\bt$.
Thus we can achieve that $\pi(\ap)=(1, 2)$. So 
\begin{equation}\label{app}
\ap=\ap'\cdot\sg_1^{-1}
\end{equation}
for some pure braid $\ap'$ on strands $1,\dots,n-1$.

Then $\pi(\bt)=(x'_1, x'_2, \ldots, x'_{n-2}, n)$ is a cycle with 
$x'_{n'}=2$ and $n>x_j'>2$ otherwise. Now, any permutation of these
$x_j'\ne 2,n$ can be realized by conjugating $\bt$ with a permutation
of $3,\dots,n-1$. Since we achieved that $\pi(\ap)$ fixes all of
these, the permutation of $x_j'\ne 2,n$ in $\pi(\bt)$ can be achieved
by a conjugation of $b=\ap\cdot\bt$, at the cost of multiplying
$\ap$ by some pure braid on strands $1,\dots,n-1$, which we can
absorb into $\ap'$ of \eqref{app}. 

This means that we can assume that we can write
\[
\bt=\bt'\cdot\bt_0\,,
\]
for some pure braid $\bt'$ on strands $2,\dots,n$, as long as $\bt_0$
is a braid on strands $2,\dots,n$ with $\pi(\bt)$ being a cycle with
$x'_{n'}=2$ when $x'_{n-1}=n$. In the following we will choose and fix
\[
\bt_0=\,\sg_3^{-1}\sg_5^{-1}\cdot\ldots\cdot\sg_{n-2}^{-1}\cdot
\sg_2^{-1}\sg_4^{-1}\cdot\ldots\cdot\sg_{n-1}^{-1}\,.
\]

\begin{lem}\label{a_3_2}
We have $a_3(L_{b_m^2})-a_3(L_{b^2})=Am^2+Bm+C$ for some (rational)
numbers $A,B,C$ (depending, \em{a priori}, on $n$ and $b$). Moreover,
$A$ does not depend on the braids $\ap'$, $\bt'$ in the presentation 
\begin{equation}\label{app3}
b=\ap'\cdot\sg_1^{-1}\cdot \bt'\cdot\bt_0\,.
\end{equation}
\end{lem}

\begin{proof}
The axis link of $L_{b_m^2}$ can be simplified similarly to
\figurename \ \ref{mtwists}. In this case we involve the
up going strand also on the left of $b$. Now we can cancel the
full twists on $n-2$ strands in $b_m^2$ by creating pairs of bands
that circle, in the opposite way, around the middle $n-2$ strands.
See \figurename \ \ref{mtwists2}. It shows the case $n=7$ and $m=1$.
(One of the pairs of circling bands, the one at the bottom, untangles,
so we have 3 such pairs.) We indicate the braids $\ap'$ and $\bt'$
just by a dashed line, showing where they have to be inserted.

\begin{figure}[htbp]
\begin{center}
  \includegraphics*[scale=0.9]{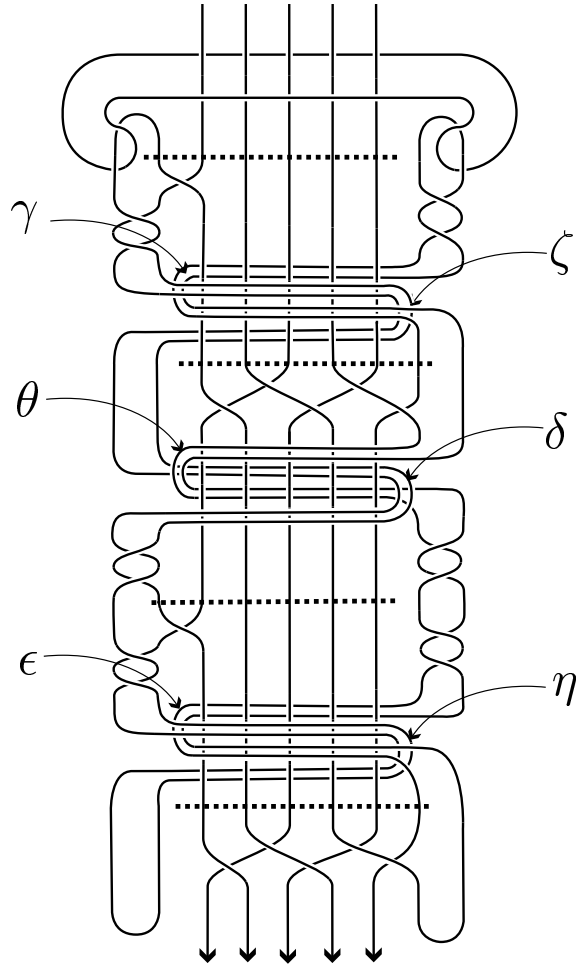}
\end{center}
\caption{$L_{b_m^2}$ (for $m=1$)}
\label{mtwists2}
\end{figure}

Now the bands $\dl$ and $\zeta$ cancel, and $\eta$ trivializes. Then,
$\th$ and $\eps$ cancel by a half-turn (and all their internal twists 
cancel), but to cancel them
further, we need to move the band past $\dl\zeta$ in the encircled
region of \figurename \ \ref{mtwists3}. (For general $m$, the parts
$\delta$ and $\zeta$ will have $|m|-1$ full turns of the band around
% IN THE PICTURE NOW $\delta$ AND $\zeta$  ARE $a$ AND $b$.
the other $n-2$ strings in the opposite direction.

\begin{figure}[htbp]
\begin{center}
  % \fbox{\includegraphics*[scale=0.7]{eps/diagram1.eps}}
  \mbox{\includegraphics*[scale=0.7]{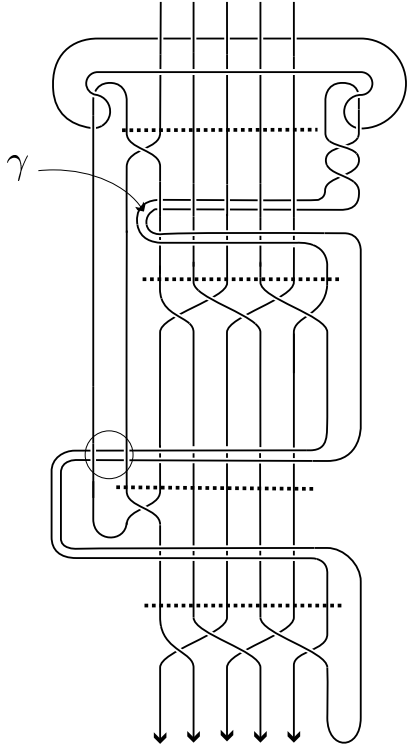}}
\end{center}
\caption{Simplified $L_{b_m^2}$ (for $m=1$)}
\label{mtwists3}
\end{figure}

Next, $\gm$ can be deleted
at the cost of changing $a_3$ by a quantity linear in $m$
(whose linear terms may depend on $n$, $\ap'$ and $\bt'$).
This can be seen from lemma \reference{a_3}, in the way
we applied it in \S\reference{S4}. It must be realized
that, in spite of the bands $\delta$ and $\zeta$ in the lower part of
the figure, the linking number of $l_1^i$ and $l_2^i$ with
$l_3^i$ does not depend on $m$. Thus the change of $a_3$
under undoing one full twist of $\gm$ does not depend on $m$
either.

This means that, for the purpose of proving lemma \ref{a_3_2}, we
can disregard the band $\gm$, and so we assume that it is trivial. 
Then we obtain from $L_{b_m^2}$ the links $K_m$ as shown (for $m=1$
and $n=7$) in \figurename \ \reference{mtwists4}. 

\begin{figure}[htbp]
% \begin{center}
\[
\begin{array}{cc}
  \mbox{\vcbox{\includegraphics*[scale=0.8]{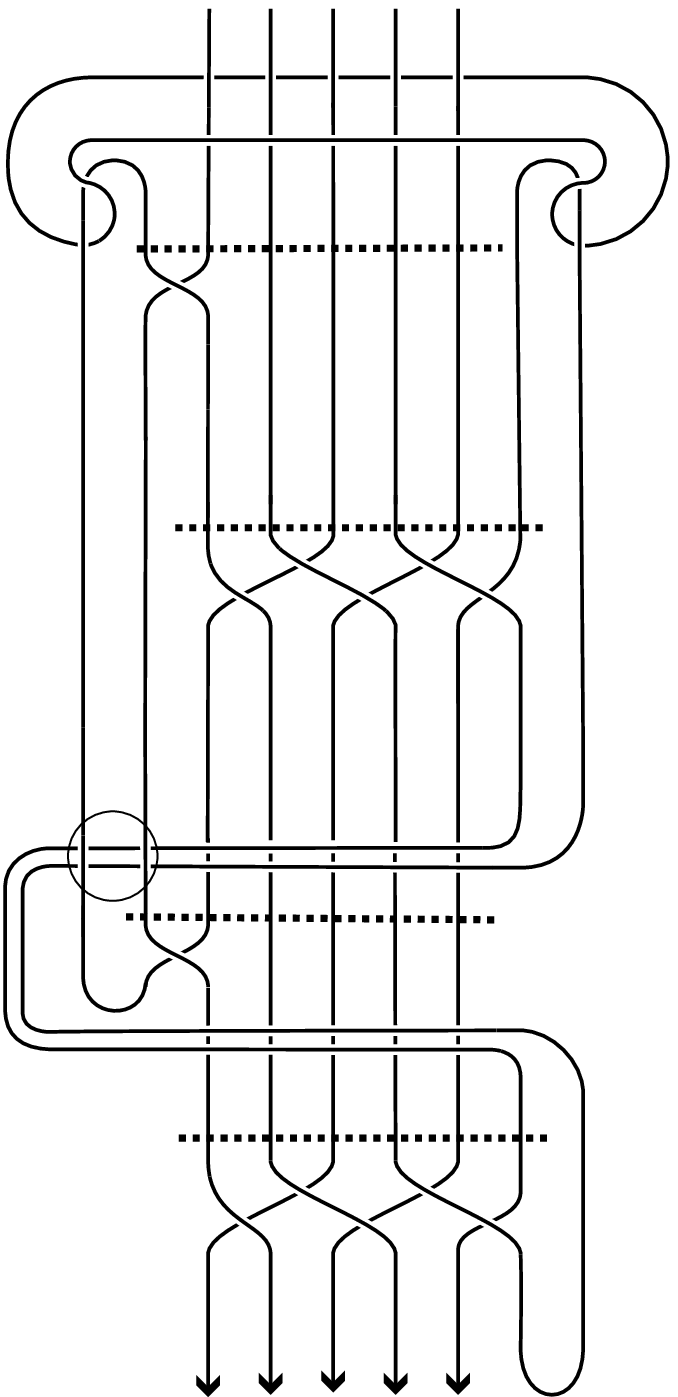}}} &
  \mbox{\vcbox{\includegraphics*[scale=0.8]{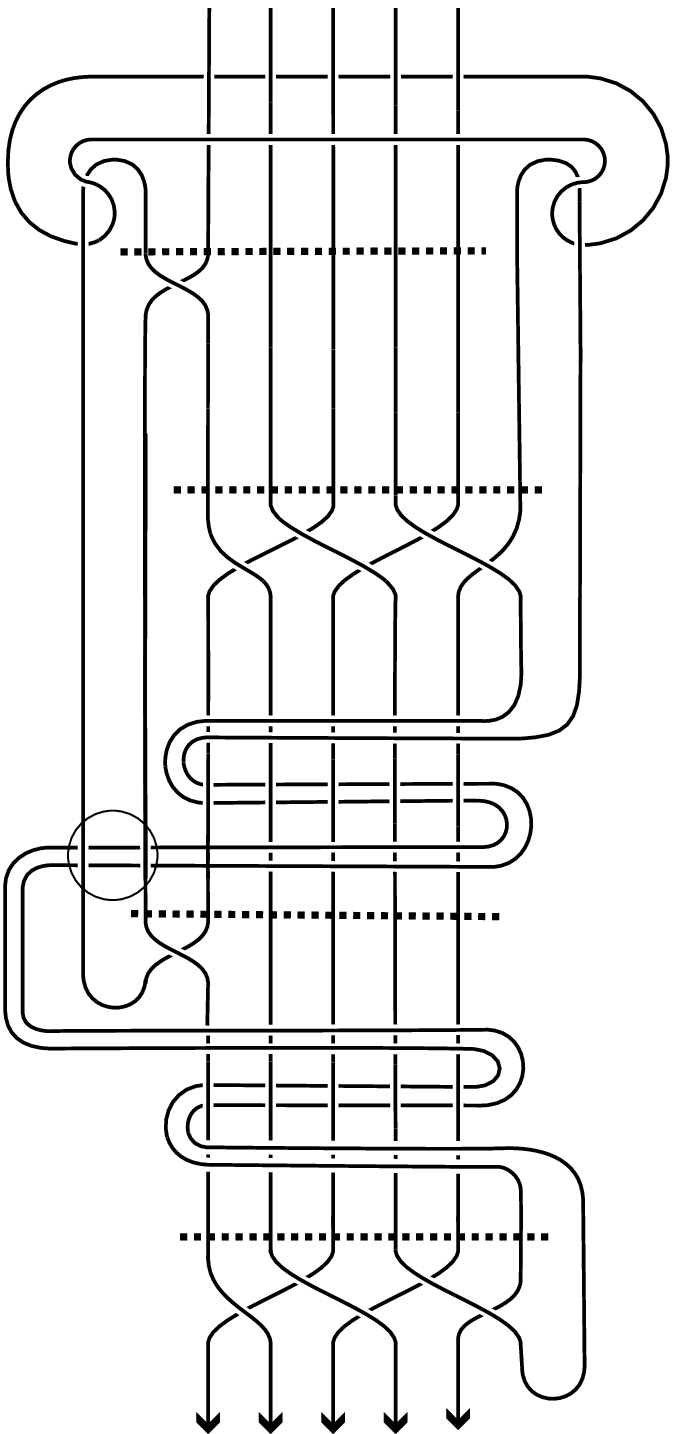}}} \\
  \ry{1.3em}K_{1} & K_{2}
%\mbox{ (part)}
\end{array}
\]
% \end{center}
\caption{$K_{m}$}
\label{mtwists4}
\end{figure}

Using the relation \eqref{nbs}, we can write
\begin{equation}\label{app2}
a_3(K_{m+1})-a_3(K_{m})\,=\,a_2(L_1)-a_2(L_2)-a_2(L_3)+a_2(L_4)\,,
\end{equation}
where $L_i=L_{m,i}$ are 3-component links obtained from $K_{m+1}$
by changing some and smoothing exactly one of the 4 crossings in
the encircled part.

Now, it is easy to observe that among the 3 linking numbers between
the components of each $L_i$, only one (the one not involving the
braid axis) depends, linearly, on $m$ (a dependence which holds
for either signs of $m$), and $\ap'$ and $\bt'$ affect
all 3 linking numbers only by some additive constant.

It follows then from theorem \reference{th} that 
$a_3(K_{m+1})-a_3(K_{m})$ is a linear expression in $m$ with a linear
term independent on $\ap'$ and $\bt'$. By inductive iteration, we
obtain the claim of lemma \reference{a_3_2}.
\end{proof}

With lemma \reference{a_3_2}, for the proof of proposition
\reference{l13}, it is legitimate to assume that $\ap'$ and
$\bt'$ are trivial, and \eqref{app3} becomes
\[
\bt_0=\,\sg_1^{-1}\sg_3^{-1}\cdot\ldots\cdot\sg_{n-2}^{-1}\cdot
\sg_2^{-1}\sg_4^{-1}\cdot\ldots\cdot\sg_{n-1}^{-1}\,.
\]
It is not difficult to evaluate the
quadratic coefficient $A$ in the lemma for this special case.

Now,
\[
2A\,=\,a_3(K_1)+a_3(K_{-1})-2a_3(K_0)\,,
\]
and so it is enough to show that the r.h.s.\ of this
equation, call it $D_n$, does not vanish for any odd $n\ge 5$.

Again, one can express $a_3(K_1)-a_3(K_0)$ and
$a_3(K_0)-a_3(K_{-1})$ using \eqref{app2}. 
Next, observe that, essentially because the action of
$\pi(b^2)$ on intermediate strands is to shift by $4$
to left or right, the replacement of any odd $n\ge 5$ by
$n+4k$ alters the component linking numbers of the
links $L_i$ in \eqref{app2} by multiples of $k$.

It follows then that $D_{n+4k}$ is a certain quadratic
expression in $k\ge 0$ for $n=5$ and $n=7$. To determine
these expressions, one can make a direct calculation
using \eqref{app2} and theorem \reference{th}. This is,
however, somewhat tedious and error-prone. Thus we
used also a different method for verification.

We drew, as in \figurename \ \ref{mtwists5}, the links $K_{\pm 1}$
and $K_0$ for $k=0,1,2$ in either case (i.e. $n=5,7,\dots,15$),
and calculated $c_i=a_3(K_i)$ by computer. 

\begin{figure}[htbp]
\begin{center}
\includegraphics*[scale=0.4]{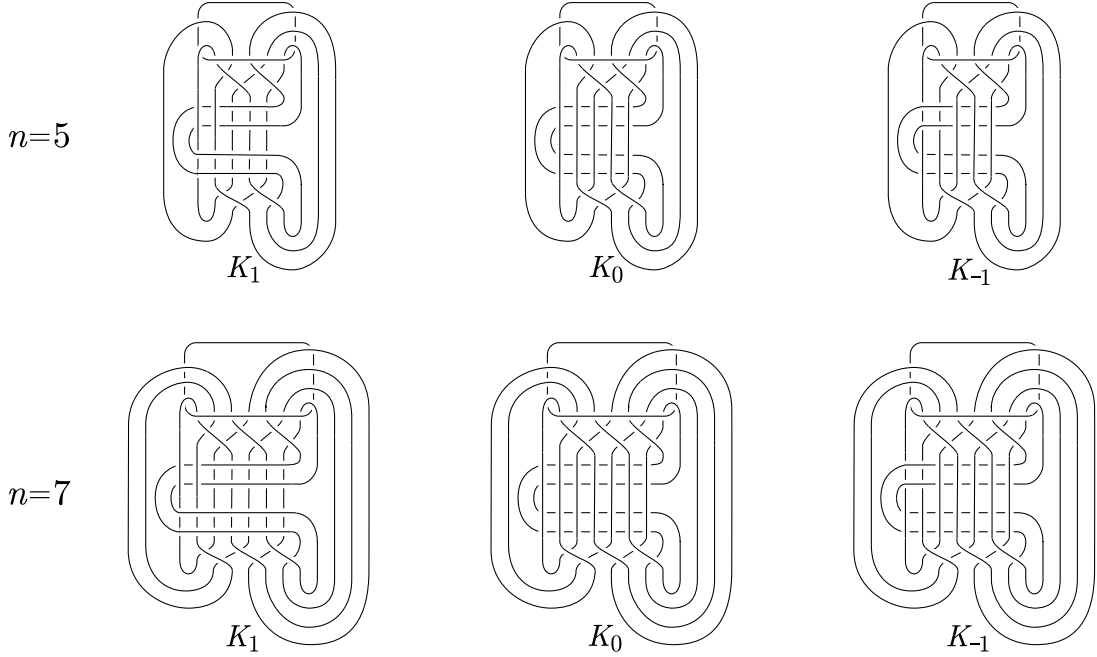}
\end{center}
\caption{$K_{m}$ for $|m|\le 1$ and $n=5,7$}
\label{mtwists5}
\end{figure}

%\begin{figure}[htbp]
%\begin{center}
%\begin{tabular}{*4c}
%$n=5$\qquad & \vv{5}{-1}{1} & \vv{5}{0}{0} & \vv{5}{1}{-1}\\
%$n=7$\qquad & \vv{7}{-1}{1} & \vv{7}{0}{0} & \vv{7}{1}{-1}\\
%\end{tabular}
%\end{center}
%\caption{$K_{m}$ for $|m|\le 1$ and $n=5,7$}
%\label{mtwists5}
%\end{figure}

The most complicated diagrams have 118 crossings, but it took
a total of 10.5 seconds to evaluate $a_3$ on all 18 diagrams using
the skein polynomial truncation algorithm of \cite{St3}.
The result is shown below:\\[-0.3mm]
%
% \begin{center}
\[
\begin{mytab}{|c||r|r|r||r|r|r|}{&&&&&&}
\hline
\Y $n$      & \mc5   & \mc9   &\mC{13}& \mc7 &\mc{11}&\mc{15} \\[1.5mm]
\hline
\hline
\Y $c_{1}$  & $-13$  & $-38 $ & $-59$ & $38$ & $137$ & $312$  \\[1mm]
   $c_{0}$  & $5  $  & $30 $  & $91 $ & $14$ & $55$  & $140$  \\[1mm]
   $c_{-1}$ & $-17$  & $-46$  & $-71$ & $46$ & $149$ & $328$  \\[1.5mm]
\hline
\end{mytab}
\]
% \end{center}

{}From this one determines that
\[
D_n\,=\,\left\{\,\begin{array}{c@{\quad\mbox{if}\ }l}
-40-72k-32k^2 & n=5+4k \\[0.5mm]
56+88k+32k^2 & n=7+4k
\end{array}\,\right..
\]
This is never zero for any $k\ge 0$. (It vanishes, however, 
for $k=-1$, which is in nice accordance with the triviality of
the cases $n=1,3$.) With this the proof of proposition
\reference{l13}, and therefore also of theorem \ref{main'}
\emph{for knots}, is concluded.
\end{proof}

\section{The first case of links}

We now move to the case of links in theorem \ref{main'}. 
A few of the links can be dealt with by a sublink argument,
but the situation seems more complicated in general.
We split the treatment of links into two major cases,
depending on whether $1$ and $n$ belong to the same or to
distinct cycles of $\pi(b)$.

\begin{thm}\label{th_l}
Assume a braid $b\in B_n$ admits an exchange move, and $1$ and $n$
belong to the same cycle of $\pi(b)$. Then the link $\hat b$
has infinitely many non-conjugate $n$-braid representations.
\end{thm}

The following is an analogue of lemma \reference{a_3}.

\begin{lem}\label{a_3'}
Let $a_{[k]}(L)=a_{n(L)+k}(L)$, with $n(L)$ being the number of
components of $L$. We have then
\[
a_{[1]} \left(\hspace{3pt}%
\mbox{\vcbox{%
\includegraphics*[scale=1.3]{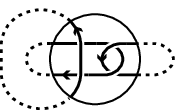}%
}}\hspace{2pt}\right) 
=
a_{[1]} \left(\hspace{3pt}\mbox{\vcbox{%
\includegraphics*[scale=1.3]{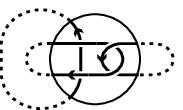}%
}}%
\hspace{2pt}\right)\,,
\]
where we allow further components to be placed (entirely) outside
the encircled spot.
\end{lem}

\begin{proof}
By switching the negative crossings on the strands in the tangle on
either side, we see that the claimed equality is equivalent to
\[
a_{[1]} \left(\hspace{3pt}%
\mbox{\vcbox{%
\includegraphics*[scale=1.3]{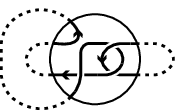}%
}}\hspace{2pt}\right) 
=
a_{[1]} \left(\hspace{3pt}\mbox{\vcbox{%
\includegraphics*[scale=1.3]{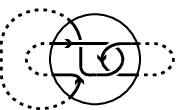}%
}}%
\hspace{2pt}\right)\,.
\]

By switching one positive crossing in the clasp on either side,
we see that this is in turn equivalent to
\[
a_{[-1]} \left(\hspace{3pt}%
\mbox{\vcbox{\includegraphics*[scale=1.3]{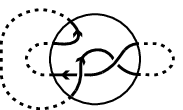}%
}}\hspace{2pt}\right) 
=
a_{[-1]} \left(\hspace{3pt}\mbox{\vcbox{%
\includegraphics*[scale=1.3]{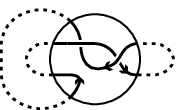}%
}}\hspace{2pt}\right).
\]
This now follows from theorem \reference{th}, since the linking
numbers of all components are the same on either side.
\end{proof}

\begin{proof}[Proof of theorem \reference{th_l}]
It is easy to see from the shape in \figurename \ \ref{braid}
that the cycle $C$ of $\pi(b)$ containing $1$ and $n$ cannot be
a transposition. If it is has length $>3$, then looking at a sublink
of $L_{b_m}$ or $L_{b_m^2}$ and using the argument in the proof of
theorem \reference{main'} for knots, we are done. So assume that
$C$ is of length 3.

We will choose a subbraid of $b$ by taking the strands corresponding
to elements in $C$ and one other cycle of $\pi(b)$. We can choose
this cycle $C'$ arbitrarily, and forget about the other components of
$\hat b$. It is enough to show that the so constructed $b_m$
are non-conjugate.

One can see with the help of lemma \ref{a_3'} that 
$a_4$ will not be helpful in distinguishing $L_{b_m}$, and we
turn to $L_{b_m^2}$. Now, in the case of $n$ odd, $C'$ is an even
(length) cycle, and we choose again one of the two components in the
closure of the subbraid $b'^2$ of $b^2$ whose permutation is $C'^2$.
This requires to treat the cases $n$ even and
odd with a little difference. Let $K'$ be the component of $\hat b'^2$
for $n$ even, or the one chosen component for $n$ odd.
And let $L'_{b^2_m}$ be the result of deleting the one component
in $L'_{b^2_m}$ for $n$ odd and $L_{b^m_2}$ for $n$ even.

We then use the argument in the proof of proposition \ref{l13},
which must be modified followingly.

Similarly to \eqref{app3}, we can achieve the form
\begin{equation}\label{5.0}
b\,=\,\ap'\cdot\sg_1\cdot \bt'\cdot\bt_0\,,
\end{equation}
with $\bt_0$ being now a `band braid' between strands $2$ and $n$
\begin{equation}\label{5.7}
\bt_0\,=\,\sg_2\cdot\ldots\cdot\sg_{n-2}\cdot
\sg_{n-1}\cdot\sg_{n-2}^{-1}\cdot\ldots\cdot \sg_2^{-1}\,.
\end{equation}
Next, the bands of $L'_{b^2_m}$ can be eliminated and ignored
(for $\gm$ now with the additional help of lemma \ref{a_3'})
as before, except that the band switch between $\th\eps$ and
$\dl\zeta$ in \figurename \ \reference{mtwists3} requires a more
careful analysis.

Now the links $L_{1}$ to $L_4$ (for fixed $m$) on the right of
\eqref{app2} have 4 components (and indices of $a_*$ have shifted
up by one). Let (for fixed $m$ and $i$) $K_1$ and $K_2$ be the
components of $L_i$ at the smoothed crossing, $K_0$ be the axis,
and $K_3=K'$ the other component (coming from the second cycle in
$\pi(b)$). Then the graph of linking numbers of $L_i$ looks like:
\begin{equation}\label{2}
\diag{1cm}{4.5}{4.5}{
 \small
 \pictranslate{1 1}{
  \point{0 0}
  \point{0 3}
  \point{3 0}
  \point{3 3}
  \picline{0 0}{3 3}
  \picline{3 0}{3 3}
  \picline{0 3}{3 3}
  \picline{0 3}{0 0}
  \picline{3 0}{0 0}
  \piccirclearc{1.2 1.2}{2.2}{125 325}
  \pictranslate{1.4 1.78}{
    \picrotate{45}{
      \picputtext{0 0}{$m+*$}
    }
  }
  \picputtext{-0.7 -0.7}{$n-3$}
  \pictranslate{-0.3 1.5}{
    \picrotate{90}{
      \picputtext{0 0}{$\pm m(n-3)+*$}
    }
  }
  \picputtext{1.5 3.3}{$\mp m(n-3)+*$}
  \picputtext{2.8 1.5}{$x$}
  \picputtext{1.5 0.25}{$3-x$}
  \picputtext{0.3 -.3}{$K_2$}
  \picputtext{3.4 3.1}{$K_1$}
  \picputtext{3.4 0.1}{$K_0$}
  \picputtext{-.3 3.3}{$K_3$}
 }
}
\end{equation}
Herein $0\le x\le 3$ is independent on $m$ and $n$, and `$*$'
means a (possibly different at every occurrence) term of the sort
\[
\ap_{0}+\sum_{k=1}^6\,\ap_{k}\lm_k\,,
\]
where $\ap_i$ are independent on $m$ and $n$, and $\lm_i$
are certain linking numbers in $\ap'$ and $\bt'$, which we
will specify shortly.

Let $s_1,\,s_2,\,s_n$ be the
strands $1,2,n$ in $b$ in the parts which enter in $\ap'$ and $\bt'$.
\[
\quad\diag{8mm}{3}{3.2}{
  \picvecline{0.5 2.4}{0.5 0}
  \picvecline{1.0 2.4}{1.0 0}
  \picvecline{2.5 2.4}{2.5 0}
  \picputtext{1.75 2.2}{.\es.\es.}
  \picputtext{1.75 0.2}{.\es.\es.}
  \picputtext{0.5 2.8}{$s_1$}
  \picputtext{1.0 2.8}{$s_2$}
  \picputtext{2.5 2.8}{$s_{n-1}$}
  \picfilledbox{1.5 1.2}{3 1.2}{$\ap$}
}
\qquad
\quad\diag{8mm}{3}{3.2}{
  \picvecline{0.5 2.4}{0.5 0}
  \picvecline{2.5 2.4}{2.5 0}
  \picputtext{0.5 2.8}{$s_2$}
  \picputtext{2.5 2.8}{$s_{n}$}
  \picputtext{1.5 2.2}{.\es.\es.}
  \picputtext{1.5 0.2}{.\es.\es.}
  \picfilledbox{1.5 1.2}{3 1.2}{$\bt$}
}
\]
(That is, the strand numbering is given at the place of $\ap'$ and
$\bt'$, and $s_2$ in $\ap'$ and $\bt'$ may be different strands of $b$.)
There are 6 types of linking numbers referred to above:
\[
\begin{array}{ccc}
\lm_1\,:=\,lk(s_1,K')  \mbox{ in }\ap'\,, & 
\lm_2\,:=\,lk(s_1,s_2) \mbox{ in }\ap'\,, & 
\lm_3\,:=\,lk(s_2,K')  \mbox{ in }\ap'\,, \\[0.4mm]
\lm_4\,:=\,lk(s_2,K')  \mbox{ in }\bt'\,, & 
\lm_5\,:=\,lk(s_n,K')  \mbox{ in }\bt'\,, &
\lm_6\,:=\,lk(s_2,s_n) \mbox{ in }\bt'\,. \\
\end{array}
\]
Here $K'$ means the strands of $b$ closing to $K'$ in the
parts within $\ap'$ and $\bt'$, and the linking number is
as explained in \S\reference{Bc}.

One can easily conclude from \eqref{2}, and because the signs of
$a_3(L_i)$ on the right of \eqref{app2} are opposite in pairs,
that $a_4(L'_{b^2_m})-a_4(L'_{b^2})$
for fixed $n$, $\ap'$ and $\bt'$ is a quadratic polynomial in $m$,
with a quadratic term of the form
\[
\sum_{j=0}^2\,\Bigl(\,\ap_{j,0}+\sum_{k=1}^6\,\ap_{j,k}\lm_k\,\Bigr)
\,n^j\,,
\]
where $\ap_{j,k}$ do not depend on $m$ and $n$.

One can determine $\ap_{j,k}$ by explicit calculation of
$a_4(L'_{b^2_m})$ for small $n$ and $m$. We have to distinguish between
even and odd $n$. We have then to realize seven 6-tuples
$(\lm_k)_{k=1}^6=(\dl_{k,l})_{k=1}^6$
(with Kronecker's delta) for $l=0,\dots,6$ by simple braids $\ap'$ and
$\bt'$, and for $n$ odd take care how the component deletion
between $L_{b^2_m}$ and $L'_{b^2_m}$ affects
these braids (the result is no longer always a braid square).

We need to take 3 different $n$ of either parity, and $|m|\le 1$,
but we calculated many additional links for consistency checks. The
outcome of this calculation is, with $\lm:=\lm_1+\lm_3+\lm_4+\lm_5$,
\begin{equation}\label{__}
[a_4(L'_{b_m^2})]_{m^2}\,=\,\left\{
\,\begin{array}{l@{\,}ll}
2(k+1)^2  & \mathbin{+} (4+5k+k^2)  \lm & \mbox{if \ }n=5+2k \\[0.6mm]
2(2k+1)^2 & \mathbin{+} (8+20k+8k^2) \lm & \mbox{if \ }n=4+2k \\
\end{array}
\right.\,.
\end{equation}

Now we are done, unless this term becomes zero for
some integer value of $\lm$. (By asymptotics, this
cannot occur for large $n$, but it does occur for $n=9$.)

To get disposed of these final cases, we consider the
mirrored braids of $b_m$. (Or alternatively, we
reverse the orientation of the braid axis.) This mirrors
the braids $\ap'$ and $\bt'$, but up to a correction
factor needed to restore the shape \eqref{5.0} with
$\bt_0$ in \eqref{5.7}. In total, mirroring $b$ changes
$\lm_i$ to $c_i-\lm_i$, where $c_1=c_3=0$, $c_2=c_5=c_6=-1$
and $c_4=1$. (The $c_i$ are the linking numbers
of the strands within these correction factors.)
Therefore, $\lm$ just changes sign.

Now for either mirroring (or orientations of the braid 
axis) the expression in \eqref{__} vanishes, only if
the absolute term in $\lm$ does so. But this
is clearly never the case. Thus up to mirroring we achieve
the desired distinction, and the proof is concluded. 
\end{proof}

\section{The second case of links}

The situation when $1$ and $n$ belong to distinct cycles of
$\pi(b)$ is the final case needed to complete the proof of theorem
\reference{main'}.

\begin{thm}\label{tfc}
Let $b\in B_n$ admit an exchange move, and let $1$ and $n$
belong to distinct non-trivial cycles of $\pi(b)$. Then infinitely many
of the $b_m$ are non-conjugate.
\end{thm}

\proof
Let $n_1$ be the length of the cycle of $\pi(b)$ containing $1$,
and $n_2$ the length of the cycle containing $n$.

By a sublink argument, and by adjusting the permutations of
the cycles involving $1$ and $n$, it is enough to consider 
$b$ in \figurename \ \reference{braid}, where $\alpha, \beta$
are given by 
\begin{equation}\label{repa}
\ap\,=\,\sg_1\cdot\ldots\cdot\sg_{n_1-1}\cdot\ap'\quad
\mbox{and}
\quad\bt\,=\,\sg_{n_1+1}\cdot\ldots\cdot\sg_{n-1}\cdot\bt'\,,
\end{equation}
and $\ap'$ and $\bt'$ are pure braids. In particular, $n_1+n_2=n$,
that is, $\pi(b)$ has only the two relevant cycles.

We will evaluate $a_4(L_{b_m})$ for fixed $\ap$ and $\bt$
as a (polynomial) function in $m$. (Note that $L_{b_m}$ is
a 3-component link.)
Let us from the outset take the attitude that the linear and
absolute term in $m$ are irrelevant.

Throughout the treatment of this final case,
we use the description of the exchange move in \eqref{alexm}.

\begin{lem}\label{lemm1}
The function $m\mapsto a_4(L_{b_m})$ is a cubic polynomial in
$m$. The cubic term does not depend on $\ap,\bt$.
The quadratic term depends on $\ap,\bt$
only via linear combinations of linking numbers of strands in
$\ap',\bt'$.
\end{lem}

\proof It is enough to work with $m>0$. Otherwise we can
multiply $\ap$ and $\bt$ by a proper power of $\kp$.
The argument we give below for $m>0$ applied on the modified $\ap$
and $\bt$ will give the result for the original $\ap$ and $\bt$
for $m>-k$, where $k$ can be chosen arbitrarily. Thus the
property holds then for all integers $m$.

We describe a method for doing a recursive skein calculation
of $a_4(L_{b_m})$, which will be relevant also after the proof
of the lemma. This calculation will be crucial throughout the
treatment, and we will gradually refine it.

We consider $a_4(L_{b_m})-a_4(L_{b_{m-1}})$, where by \eqref{alexm}
\[
b_m=\ap\kp^m\bt\kp^{-m}\,.
\]
Now we can write
\begin{eqnarray}\label{a12}
b_m & = & \ap\kp^{m-1}(\sg_1\cdot\ldots\cdot\sg_{n-2}\sg_{n-1})
(\ul{\sg_{n-1}^{-1}}\sg_{n-2}\cdot\ldots\cdot\sg_1)\bt \\
\nonumber & = & (\sg_1^{-1}\cdot\ldots\cdot\sg_{n-2}^{-1}\sg_{n-1}^{-1})
(\ul{\sg_{n-1}}\sg_{n-2}^{-1}\cdot\ldots\cdot\sg_1^{-1})\kp^{1-m}\,.
\end{eqnarray}
Then we have by the skein relation \eqref{Lpm0}
\begin{equation}\label{one}
a_4(L_{b_m})-a_4(L_{b_{m-1}})\,=\,-a_3(L_{m-1,1})+a_3(L_{m-1,2})\,,
\end{equation}
where $L_{m-1,i}$ is the axis link of the braid obtained from
the word on the right of \eqref{a12} by omitting the underlined
occurrences of $\sg_{n-1}^{-1}$ resp.\ $\sg_{n-1}$. Let us write
$[b]$ for $L_b$. Then
\begin{eqnarray}
L_{m,1} & = &\label{m1br}
[\ap\kp^{m}(\sg_1\cdot\ldots\cdot\sg_{n-2}\sg_{n-1})
(\sg_{n-2}\cdot\ldots\cdot\sg_1)\bt \\
\nonumber
 & & (\sg_1^{-1}\cdot\ldots\cdot\sg_{n-2}^{-1}\sg_{n-1}^{-1})
(\sg_{n-1}\sg_{n-2}^{-1}\cdot\ldots\cdot\sg_1^{-1})\kp^{-m}]\,.
\\
L_{m,2} & = &\label{m2br}
[\ap\kp^{m}(\sg_1\cdot\ldots\cdot\sg_{n-2}\sg_{n-1})
(\sg_{n-1}\sg_{n-2}\cdot\ldots\cdot\sg_1)\bt\\
\nonumber
 & & (\sg_1^{-1}\cdot\ldots\cdot\sg_{n-2}^{-1}\sg_{n-1}^{-1})
(\sg_{n-2}^{-1}\cdot\ldots\cdot\sg_1^{-1})\kp^{-m}]\,.
\end{eqnarray}

The complication now is that the links $L_{m,1}$ have
two components. We need to apply the skein relation once more
before we can use Hoste's formula.

We will calculate instead of $a_3(L_{m-1,i})$ the
difference
\begin{equation}\label{d12}
a_3(L_{m-1,i})-a_3(L_{0,i})\,.
\end{equation}
The extra terms $a_3(L_{0,i})$ contribute
only something absolute in $m$ to $a_4(L_{b_m})-a_4(L_{b_{m-1}})$,
and hence only something linear in $m$ to $a_4(L_{b_m})$,
which we decided to ignore.

It is clear that one can determine \eqref{d12} by evaluating 
\[
a_3(L_{m,i})-a_3(L_{m-1,i})\,.
\]
For this we turn around two groups of $n-2$ crossings, namely
those needed to trivialize the last of the $m$ copies of $\kp$
before $\bt$ in \eqref{a12} (note that we shifted $m-1$ to $m$)
and the first of the $m$ copies of $\kp^{-1}$ after $\bt$.
We obtain
\begin{equation}\label{two}
a_3(L_{m,i})-a_3(L_{m-1,i})=\sum_{l=2}^{n-1}a_2(L_{m,i,l})-
a_2(L_{m,i,\bar l})\,.
\end{equation}
The link $L_{m,i,l}$ is the axis link of the braid
obtained from the braid in \eqref{m1br} (for $i=1$) or 
\eqref{m2br} (for $i=2$) by replacing the
last copy of $\kp$ before $\bt$ by 
\begin{equation}\label{qb}
\sg_1\dots\sg_{l-2}\sg_{l-1}\sg_{l-2}\dots\sg_1\,,
\end{equation}
and $L_{m,i,\bar l}$ is the axis link of the braid
obtained from the braid in \eqref{m1br} resp.\ \eqref{m2br}
by replacing the
first copy of $\kp^{-1}$ after $\bt$ by the
inverse of the braid in \eqref{qb}.

Now $L_{m,i,l}$ and $L_{m,i,\bar l}$ have three components, and their
$a_2$ can be evaluated by Hoste's formula.

Two of the linking numbers of the components of $L_{m,i,l}$ and
$L_{m,i,\bar l}$ are independent of $m$, and the third one
is linear in $m$, with linear term independent of $\ap'$,
$\bt'$. From this the claim of the lemma follows. \qed

\begin{lem}\label{lemm2}
In the function $m\mapsto a_4(L_{b_m})$ of lemma \ref{lemm1}, the
cubic term vanishes.
\end{lem}

\proof By lemma \ref{lemm1}, it is enough to prove this when $\ap'$
and $\bt'$ are trivial. Under this assumption, we claim the
following:
\begin{equation}\label{xm}
L_{b_m}\simeq L_{b_{-m}}
\end{equation}
up to switching orientation (of \em{all} components simultaneously).
With \eqref{xm} the lemma follows, since the function given there
is even.

To see \eqref{xm}, note that 
\[
\ap=\sg_1\dots\sg_{n_1-1}
\]
can be conjugated to its word-reverse \em{without using} $\sg_1$
and $\sg_{n-1}$, and similarly $\bt$. Then $\kp$ commutes with
the subgroup generated by $\sg_2,\dots,\sg_{n-2}$. After $\ap$
and $\bt$ were reversed, flip the braid axis link by $\pi$ along
the horizontal axis in projection plane, conjugate by $\ap$
to move it to the top, and reverse all orientations (including of the
axis) to have strands pointing downward. \qed

We thus now are led to look at $[a_3(L_{b_m})]_{m_2}$, and our goal
is to prove that is does not vanish.
The skein calculation in the proof of lemma \reference{lemm1}
would be unwieldy.  However, we help ourselves again by taking
also the mirrored braids into account. 

Let $\bar b$ be $b$ where all $\sg_i$ and $\sg_{i}^{-1}$ are
interchanged. Mirroring is an automorphism of $B_n$, so
if two braids are conjugate, so are their mirror images. 
We will thus complete the proof of theorem \reference{tfc},
and hence also of theorem \reference{main'}, by the following lemma.

\begin{lem} We have
\[
[a_4(L_{b_m})]_{m^2}+[a_4(L_{\ol{b}_m})]_{m^2}\,=\,
2(n_1-1)(n_2-1)\,.
\]
\label{lmm3}
\end{lem}

\proof 
By lemma \reference{lemm1} we have that
\[
[a_4(L_{b_m})]_{m^2}
\]
depends only linearly on the linking numbers of $\ap'$ and $\bt'$.
Now, changing a linking number in $\ap'$ for the representation 
\eqref{repa} of $\ap$ changes this linking number oppositely in the
representation of $\ol{\ap}$. We see again that the expression
\begin{equation}\label{BLA}
[a_4(L_{b_m})]_{m^2}+[a_4(L_{\ol{b}_m})]_{m^2}
\end{equation}
does not depend on $\ap$ and $\bt$. We will thus evaluate it
when $\ap$ and $\bt$ are trivial.

We have by \eqref{xm} then
\begin{equation}\label{mpr}
L_{{\ol b}_m}=L_{{\ol b}_{-m}}=L_{\ol{b_m}}\,.
\end{equation}
Now we will follow the skein calculation of the proof
of lemma \ref{lemm1}, simultaneously for $b_m$ and $\ol{b_m}$. 
In order to distinguish the links occurring in the
calculations for $b_m$ and $\ol{b_m}$, we will write
in the latter case $\ol{L}_{\dots}$, with the proper
subscript, for what would have been ${L}_{\dots}$ in the
case of $b_m$.

The skein calculation could be summarized by saying that
we expressed $a_4(L_{b_m})-a_4(L_{b_{m-1}})$ by a linear
combination of terms 
\begin{equation}\label{oi}
a_2(L_{m',i,l})\quad\mbox{ and }\quad a_2(L_{m',i,\bar l})
\end{equation}
for $i=1,2$; $0\le m'<m$; and $2\le l\le n-1$, up to
absolute terms in $m$. If we sum this up to express
$a_4(L_{b_m})$, then we have
something linear in $m$, and then for each of the 4
families in \eqref{oi}:
\begin{equation}\label{4fam}
(1)=L_{m',1,\bar l}\,,\quad (2)=L_{m',2,l}\,,\quad 
(3)=L_{m',2,\bar l}\,,\quad (4)=L_{m',1,l} 
\end{equation}
(determined by the choice $i=1,2$ and between $l$
and $\bar l$), there are 
\begin{equation}\label{oj}
\frac{m^2}{2}+O(m)\
\end{equation}
terms. 

Then each of the terms
\[
a_2(\ol{L}_{m',i,l})\quad\mbox{ and }\quad a_2(\ol{L}_{m',i,\bar l})
\]
enters into the skein calculation for $\ol{b_m}$ with the same
sign as does its analogue in \eqref{oi} for the calculation for
$b_m$. This is because every time a crossing is
smoothed out, the sign changes between $b_m$ and $\ol{b_m}$,
but to get \eqref{oi} we smoothed out two crossings in $b_m$
resp.\ $\ol{b_m}$. Combining the signs in \eqref{one} and
\eqref{two}, we see that the signs of families (1) and (2) in
\eqref{4fam} are positive, for families (3) and (4) negative.

Now, the component linking numbers of $\ol{L}_{m',i,l}$ with the
axis are the same as for ${L}_{m',i,l}$, and the remaining one
linking number is opposite. 

By Hoste's formula, it becomes clear that one half of \eqref{BLA}
can be evaluated by doing again the skein calculation for
$b_m$ only, and therein replacing $a_3(L_{\dots})$ by
\begin{equation}\label{i8}
\br{\pi(b_{\dots})}\,,
\end{equation}
where $b_{\dots}$ is the braid
whose axis link is $L_{\dots}$, and $\br{\sg}$ is the
product of the (here always two) cycle lengths of $\sg$.

Now this simplifies the calculation considerably.  Note first
that $\br{\pi(b_{\dots})}$ does not depend on $m$ or $m'$. Thus
we can evaluate all four families in \eqref{4fam} just by looking at
their permutations. We have then to divide by 2 following \eqref{oj}
to get the $m^2$-term. This can be compensated by the factor 2
explained in the application of Hoste's formula above
\eqref{i8}.

{}From here there are two ways to get done.
A "philosophical" way is to observe that by the
skein calculation, the expression \eqref{BLA}
must be some polynomial in $n_1$ and $n_2$. 
By using that $L_{b_m}$ has $O(m(n_1+n_2))$
crossings, that $a_4$ is a Vassiliev invariant
of degree $4$, and the extension of the Lin-Wang conjecture 
to links in \cite{St4}, we can conclude that the polynomial
is of degree at most $4$. Moreover, the triviality of the
cases $n_i=1$ explains the factor $(n_1-1)(n_2-1)$.
The polynomial must also be symmetric in $n_1$ and $n_2$.
{}From this one can get the formula in the lemma by
calculating the value of the polynomial for a few explicit
$(n_1,n_2)$. (In the realm of ascertaining the result, we
did a few such checks which, via this argument, would
establish the formula alternatively.)

Nevertheless, it is possible to make exact calculation. Now let
us write $(1),\dots,(4)$ in \eqref{4fam} for the contribution
\eqref{i8} of the link in question to $a_4(L_{b_m})$ according
to \eqref{one} and \eqref{two}.

Let $[x,y]$ be the cycle $(y,\,y-1,\,\dots,x)$ and $\sg\tau=\tau
\circ\sg$ be the compositive multiplication of permutations. We have
\begin{eqnarray*}
(1) & = & \br{(1,n)[n_1+1,n](1,l)[1,n_1]} \\
    & = & \br{\left(\begin{array}{cl} l+1 & \mbox{if\es}n\ge n_1+1 \\
    l & \mbox{if\es}l\le n_1\end{array},n\right)(1,n)\pi(\bt)} \\
(2) & = & \br{(1,l)[n_1+1,n](1,n)[1,n_1]} \\
    & = & \br{\left(1,\begin{array}{cl} l & \mbox{if\es}l\le n_1 \\
     n & \mbox{if\es}l=n_1+1 \\ l-1 & \mbox{if\es}l>n_1 \end{array}
      \right)(1,n)\pi(\bt)} \\
(3) & = & \br{[n_1+1,n](1,n)(1,l)[1,n_1] }\\
    & = & \br{(l,n)(1,n)\pi(b)}\,, \\
(4) & = & \br{(1,l)(1,n)\pi(b)}\,.
\end{eqnarray*}

Then for $l\le n_1$ we have $(1)=(3)$ and $(2)=(4)$, and
in the sum over $l>n_1$ of $(1)-(3)$ terms cancel with a
shift of $1$. Similarly for $(2)-(4)$\,.

We have then
\[
\sum_{l}(1)+(2)-(3)-(4)=(1)_{l=n-1}-(3)_{l=n_1+1}+
  (2)_{l=n_1+1}-(4)_{l=n-1}\,.
\]
The two permutations with positive sign are equal to $\pi(b)$, while
the other two have a fixpoint (and a cycle of length $n-1$), and
the result follows.
\qed

With lemma \reference{lmm3}, theorem \reference{tfc}, and therewith
also theorem \reference{main'}, is proved. \qed

\section{Examples, applications and problems}

As a consequence of theorem \reference{th_l}, we obtain the
following result in \cite{St2}.

\begin{cor}
Let $L$ be a composite link of braid index $b(L)\ge 4$, which
factors as $L_1\# L_2$ in such a way, that the components
of either $L_{1,2}$ the connected sum is performed at are knotted.
(E.g.\ any composite \em{knot} $L$ will do.) Then $L$ has
infinitely many non-conjugate minimal braid representations.
\end{cor}

\begin{proof}
By the 1-subadditivity of the braid index under connected sum proved
by Birman and Menasco \cite{B1}, $L$ has a composite minimal braid
representation $b$, of the sort illustrated in
\figurename \ \ref{torefoil} (where $\hat b_i=L_i$).
Such a representation admits an exchange move if it has $n=b(L)\ge 4$
strands. By assumption, the component of the common strand of $b_1$
and $b_2$ has at least one other strand in either of these. By
conjugation of $b_i$ it can be made to be strand $1$ and $n$ (in
$b$), so that the cycle condition of theorem \ref{th_l} also holds.
\end{proof}

\begin{figure}[htbp]
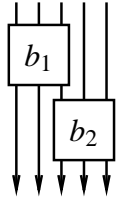

\begin{center}
\[
\diag{1cm}{2}{2.5}{
 \picvecwidth{0.04}
 \picmultigraphics{5}{0.3 0}{
  \picvecline{0.4 2.5}{0.4 0}
 }
 \picfilledbox{0.7 1.8}{0.8 d}{$b_1$}
 \picfilledbox{1.3 0.8}{0.8 d}{$b_2$}
}%\qquad\lra\qquad
\]%\includegraphics*[scale=2]{eps/trefoil.eps}
\end{center}
\caption{A composite braid}
\label{torefoil}
\end{figure}

\begin{ex}
We found that
most prime knots $K$ of crossing number $c(K)\le 10$ and $b(K)\ge 4$,
except 7 knots for $c(K)=9$ and 15 knots for $c(K)=10$, have a minimal
representation admitting an exchange move. (See the table in
\S\ref{Stb}.) So these knots have
infinitely many conjugacy classes of minimal braid representations.
\end{ex}

Note that the exchange move in \figurename \ \ref{bwitht} is trivial
when the leftmost strand of $\ap$ or the rightmost strand of $\bt$ are
isolated. We do not know if under exclusion of these obvious cases, the
move can always yield infinitely many conjugacy classes. Certainly,
theorem \ref{main'} gives the weakest condition in terms of $\pi(b)$
alone under which the exchange move can work.

If one likes to investigate the remaining braids, one must be aware
that a construction of Stanford \cite{Stan} allows one to `approximate'
these cases of failure by others which cannot be distinguished by any
number of Vassiliev invariants (including coefficients of $\nb$). With
this insight in advance, we expect little decent outcome in trying
to apply our approach in the excluded case. 

In \cite{St2} we used some Lie group approach which covers some of
these braids when in \figurename \ \ref{bwitht} we have $\bt=\sg_n^
{\pm 1}$. This approach promises no satisfactory adaptation to
exchange moves.

Apart from these difficulties, we conclude with two more remaining
problems.

\begin{prob}
Theorem \ref{main'} suggests to seek braids admitting exchange moves,
but the identification what links have such (minimal) braids is still
difficult.
\end{prob}

\begin{prob}
We do not know how to construct \em{Markov irreducible} $b\in B_n$
with $n>b(L)$, i.e.\ such not conjugate to stabilizations
$b\sg_{n-1}^{\pm 1}$ for $b\in B_{n-1}$. The only examples known, due
to Morton and Fiedler \cite{M2,Fie}, are for $n=4$ and $K=\hat b$
being the unknot.
\end{prob}

\section{Table\label{Stb}}

The below list gives 4-braid representations admitting an exchange
move for 95 knots of braid index 4 (up to mirroring) in the tables of
\cite[appendix]{Rolfsen}.
An integer $i>0$ means $\sg_i$, an $i<0$ stands for $\sg_{-i}^{-1}$.
Note that an $n$-braid word admits an exchange move up to cyclic
permutation if and only if it has no (cyclic) subword of the form 
\begin{equation}\label{fr}
\sg_1^{\pm 1}\cdot\ldots\cdot\sg_{n-1}^{\pm 1}\cdot\ldots
\cdot\sg_{1}^{\pm 1}\cdot\ldots\cdot\sg_{n-1}^{\pm 1}\ldots\,.
\end{equation}

We shifted, as in \cite{BurdeZieschang}, indices down by 1 for
$10_{163},\,\dots,10_{166}$ to discard Perko's duplication. Thus,
e.g., the knot written below as $10_{162}$ is Rolfsen's $10_{163}$.

A further mistake in Rolfsen's tables is that therein $10_{83}$ and
$10_{86}$ are swapped: the Conway notation and Alexander polynomial
for each one refers to the diagram of the other. The convention taken
here is that we interchange Conway notations and Alexander polynomials
to fix the mismatch (as in the corrected reprinting of Rolfsen's
book), and \em{not} the diagrams (as, e.g., in \cite{Kawauchi}).

{\small
\begin{multicols}{2}
\begin{tabbing}
  $10_{163}$\es \= \tt  -3 -3 -2 -2 3 -2 3 3 -1 2 -1 \kill 
  $6_{1}$ \> \tt  -3 -3 -2 1 1 2 -1 3 -2 \\
  $7_{2}$ \> \tt  -3 -3 -3 -1 -1 -2 1 3 -2 \\
  $7_{4}$ \> \tt  -3 -3 -2 -1 -1 -2 1 3 -2 \\
  $7_{6}$ \> \tt  -3 -3 -3 -1 2 -1 2 3 -2 \\
  $7_{7}$ \> \tt  -3 -3 -2 1 -2 1 3 2 2 \\
  $8_{4}$ \> \tt  -3 -3 -3 -1 2 1 1 -3 2 \\
  $8_{6}$ \> \tt  -3 -3 -3 -3 1 -2 1 3 -2 \\
  $8_{8}$ \> \tt  -3 -3 -3 1 -2 1 1 3 -2 \\
  $8_{11}$ \> \tt  -3 -3 -2 -2 1 -2 1 3 -2 \\
  $8_{13}$ \> \tt  -3 -3 -2 1 -2 1 1 3 -2 \\
  $8_{14}$ \> \tt  -3 -3 -3 -2 1 -2 1 3 -2 \\
  $8_{15}$ \> \tt  -3 -3 -2 -2 -3 -1 -1 2 -1 \\
  $9_{4}$ \> \tt  -3 -3 -3 -3 -3 -1 -1 -2 1 3 -2 \\
  $9_{7}$ \> \tt  -3 -3 -3 -3 -2 -2 3 -1 -1 -2 1 \\
  $9_{10}$ \> \tt  -3 -3 -2 -2 -2 -1 -1 -2 1 3 -2 \\
  $9_{11}$ \> \tt  -3 -3 -3 -3 -1 2 -1 -3 2 \\
  $9_{13}$ \> \tt  -3 -3 -3 -3 -2 -1 -1 -2 1 3 -2 \\
  $9_{17}$ \> \tt  -3 -1 2 -1 2 2 2 -3 2 \\
  $9_{18}$ \> \tt  -3 -3 -3 -2 -2 -1 -1 -2 1 3 -2 \\
  $9_{20}$ \> \tt  -3 -3 -3 -1 -1 2 -1 -3 2 \\
  $9_{22}$ \> \tt  -3 -1 2 -1 2 -3 2 2 2 \\
  $9_{23}$ \> \tt  -3 -3 -3 -2 -2 -1 -1 2 -1 3 -2 \\
  $9_{24}$ \> \tt  -3 -3 -1 2 -1 -3 2 2 2 \\
  $9_{26}$ \> \tt  -3 -3 -3 -1 2 -1 2 -3 2 \\
  $9_{27}$ \> \tt  -3 -3 -1 2 -1 2 2 -3 2 \\
  $9_{28}$ \> \tt  -3 -3 -1 -1 2 -1 -3 2 2 \\
  $9_{30}$ \> \tt  -3 -3 -1 2 -1 2 -3 2 2 \\
  $9_{31}$ \> \tt  -3 -3 -1 -1 2 -1 2 -3 2 \\
  $9_{32}$ \> \tt  -3 -3 -1 2 -1 -3 2 -3 2 \\
  $9_{33}$ \> \tt  -3 -1 2 -1 -3 2 -3 2 2 \\
  $9_{36}$ \> \tt  -3 -3 -3 -1 2 -1 -3 -3 2 \\
  $9_{42}$ \> \tt  -3 -3 -3 -1 2 -1 3 3 2 \\
  $9_{43}$ \> \tt  -3 -3 -3 1 -2 1 -3 -3 -2 \\
  $9_{44}$ \> \tt  -3 -3 -3 1 -2 1 3 3 -2 \\
  $9_{45}$ \> \tt  -3 -3 -2 -2 -3 -1 2 -1 2 \\
  $9_{49}$ \> \tt  -3 -3 -2 -2 -3 -3 -1 2 -1 3 -2 \\
  $10_{6}$ \> \tt  -3 -3 -3 -3 -3 -3 1 -2 1 3 -2 \\
  $10_{8}$ \> \tt  -3 -3 -3 -3 -3 -1 2 1 1 -3 2 \\
  $10_{12}$ \> \tt  -3 -3 -3 -3 -3 1 -2 1 1 3 -2 \\
  $10_{14}$ \> \tt  -3 -3 -3 -3 -3 -2 1 -2 1 3 -2 \\
  $10_{15}$ \> \tt  -3 -3 -3 -3 -1 2 1 1 -3 2 2 \\
  $10_{19}$ \> \tt  -3 -3 -3 -3 -1 2 1 1 2 -3 2 \\
  $10_{21}$ \> \tt  -3 -3 -2 -2 -2 -2 1 -2 1 3 -2 \\
  $10_{22}$ \> \tt  -3 -3 -3 -3 1 -2 1 1 1 3 -2 \\
  $10_{23}$ \> \tt  -3 -3 -2 -2 -2 1 -2 1 1 3 -2 \\
  $10_{25}$ \> \tt  -3 -3 -3 -3 -2 -2 1 -2 1 3 -2 \\
  $10_{26}$ \> \tt  -3 -3 -3 -1 2 1 1 2 2 -3 2 \\
  $10_{27}$ \> \tt  -3 -3 -3 -3 -2 1 -2 1 1 3 -2 \\
  $10_{32}$ \> \tt  -3 -3 -3 -2 1 -2 1 1 1 3 -2 \\
  $10_{39}$ \> \tt  -3 -3 -3 -2 -2 -2 1 -2 1 3 -2 \\
  $10_{40}$ \> \tt  -3 -3 -3 -2 -2 1 -2 1 1 3 -2 \\
  $10_{49}$ \> \tt  -3 -3 -3 -3 -2 -2 -3 -1 -1 2 -1 \\
  $10_{50}$ \> \tt  -3 -3 -2 -2 1 -2 -2 -2 1 3 -2 \\
  $10_{51}$ \> \tt  -3 -3 -2 -2 1 -2 -2 1 1 3 -2 \\
  $10_{52}$ \> \tt  -3 -3 -3 -1 2 1 1 2 -3 -3 2 \\
  $10_{54}$ \> \tt  -3 -3 -3 -1 2 1 1 -3 -3 2 2 \\
  $10_{56}$ \> \tt  -3 -3 -3 -2 1 -2 -2 -2 1 3 -2 \\
  $10_{57}$ \> \tt  -3 -3 -3 -2 1 -2 -2 1 1 3 -2 \\
  $10_{61}$ \> \tt  -3 -3 -3 -1 2 1 1 -3 -3 -3 2 \\
  $10_{65}$ \> \tt  -3 -3 -2 1 -2 -2 -2 1 1 3 -2 \\
  $10_{66}$ \> \tt  -3 -3 -3 -2 -2 -2 -3 -1 -1 2 -1 \\
  $10_{72}$ \> \tt  -3 -3 -3 -3 -2 -2 3 -2 1 -2 1 \\
  $10_{76}$ \> \tt  -3 -3 -3 -3 -2 -2 -2 3 1 -2 1 \\
  $10_{77}$ \> \tt  -3 -3 -3 -3 -2 -2 3 1 -2 1 1 \\
  $10_{80}$ \> \tt  -3 -3 -3 -2 -2 -3 -3 -1 -1 2 -1 \\
  $10_{83}$ \> \tt  -3 -3 -2 1 -2 -2 1 -2 1 3 -2 \\
  $10_{84}$ \> \tt  -3 -3 -3 1 -2 -2 1 -2 1 3 -2 \\
  $10_{86}$ \> \tt  -3 -3 -2 1 -2 1 -2 1 1 3 -2 \\
  $10_{87}$ \> \tt  -3 -3 -3 1 -2 1 -2 1 1 3 -2 \\
  $10_{90}$ \> \tt  -3 -3 1 -2 -2 1 1 -2 1 3 -2 \\
  $10_{93}$ \> \tt  -3 -3 -1 2 1 1 -3 2 -3 -3 2 \\
  $10_{102}$ \> \tt  -3 -3 -1 2 1 1 -3 2 -3 2 2 \\
  $10_{103}$ \> \tt  -3 -3 1 -2 -2 1 -2 -2 1 3 -2 \\
  $10_{108}$ \> \tt  -3 -3 -1 2 1 1 -3 -3 2 -3 2 \\
  $10_{114}$ \> \tt  -3 -3 1 -2 1 -2 1 -2 1 3 -2 \\
  $10_{128}$ \> \tt  -3 -3 -3 -2 -2 -3 -1 2 -1 -3 -2 \\
  $10_{129}$ \> \tt  -3 -3 -3 2 2 3 -1 2 -1 3 2 \\
  $10_{130}$ \> \tt  -3 -3 -3 -1 2 1 1 2 3 3 2 \\
  $10_{131}$ \> \tt  -3 -3 -3 -2 -1 -1 -2 1 3 3 -2 \\
  $10_{132}$ \> \tt  -3 -3 -3 -1 -1 -2 1 1 1 3 -2 \\
  $10_{133}$ \> \tt  -3 -3 -3 -2 -2 3 -1 2 -1 3 -2 \\
  $10_{134}$ \> \tt  -3 -3 -3 -2 -2 -3 -3 -1 -1 -2 1 \\
  $10_{135}$ \> \tt  -3 -3 -3 -2 1 1 2 2 1 3 -2 \\
  $10_{140}$ \> \tt  -3 -3 -3 -1 -1 -2 1 3 3 3 -2 \\
  $10_{142}$ \> \tt  -3 -3 -2 -2 -3 -3 -1 2 -1 -3 -2 \\
  $10_{144}$ \> \tt  -3 -3 -2 -2 -3 -1 2 1 1 -3 2 \\
  $10_{150}$ \> \tt  -3 -3 -3 -2 -2 -2 3 3 -1 2 -1 \\
  $10_{151}$ \> \tt  -3 -3 -3 -1 2 -1 2 2 -1 3 -2 \\
  $10_{153}$ \> \tt  -3 -3 -3 1 -2 1 -3 -3 2 2 2 \\
  $10_{154}$ \> \tt  -3 -3 -2 -2 -3 -1 -2 -2 -1 -1 2 \\
  $10_{156}$ \> \tt  -3 -3 -3 -1 2 -1 3 3 -2 3 -2 \\
  $10_{158}$ \> \tt  -3 -3 -3 1 -2 1 3 -2 3 2 2 \\
  $10_{160}$ \> \tt  -3 -3 -3 -2 -2 -3 2 -3 1 -2 1 \\
  $10_{162}$ \> \tt  -3 -3 -2 -2 -3 2 -3 -1 2 1 1 \\
  $10_{163}$ \> \tt  -3 -3 -2 -2 3 -2 3 3 -1 2 -1 \\
\end{tabbing}
\end{multicols}
}

The 22 knots of braid index 4 for which we could not find a 4-braid
admitting an exchange move are:

\begin{tabbing}
\kern0.17\textwidth\= \kern0.17\textwidth\= \kern0.17\textwidth\=
\kern0.17\textwidth\= \kern0.17\textwidth\= \kern0.17\textwidth
\=\kill
$9_{29}$,\> $9_{34}$,\> $9_{38}$,\> $9_{40}$,\> $9_{46}$,\> $9_{47}$,\\
$9_{48}$,\> $10_{92}$,\> $10_{95}$,\> $10_{98}$,\> $10_{111}$,
  \> $10_{113}$,\\
$10_{117}$,\> $10_{119}$,\> $10_{121}$,\> $10_{122}$,\> $10_{136}$,
  \> $10_{145}$,\\
$10_{146}$,\> $10_{147}$,\> $10_{164}$,\> $10_{165}$\,.
\end{tabbing}

Among the 5- and 6-braid knots in the table, all have a minimal
braid admitting an exchange move. This can be checked from the
representations given in \cite{St5}. It is easy to see that the
(cyclic) shape \eqref{fr} (with $n=5,6$) can be avoided after
possibly applying some commutativity relations.

\begin{ack}
This work was carried out under support by the Brain Korea 21 (BK21)
Project of the Ministry of Education \& Human Resources Development of
the Republic of Korea. The authors are grateful to KAIST, Daejeon,
Korea for its hospitality.
%The author gratitude to Professors Kouki Taniyama and
% Akira Yasuhara for their valuable advice and encouragement.
%I express my deep appreciation to Professor Mikami Hirasawa 
%for helpful discussion and advice.
\end{ack}

\end{document}